\documentclass[a4paper]{article}
\usepackage[utf8]{inputenc}
\usepackage[english]{babel}
\usepackage[T1]{fontenc}
\usepackage{amsmath,amssymb}

\usepackage{hyperref}

\renewcommand{\bar}{\overline}
\renewcommand{\tilde}{\widetilde}

\usepackage{microtype}

\usepackage{xspace}

\usepackage[all,line]{xy}

\usepackage{mathtools}

\newcommand{\spa}[2]{\operatorname{span}_{#1}\ensuremath{\left\{#2\right\}}}

\newcommand{\inv}{^{-1}}

\newcommand{\iso}{\cong}

\DeclareMathOperator{\wt}{wt}

\DeclareMathOperator{\ch}{ch}

\DeclareMathOperator{\ad}{ad}

\DeclareMathOperator{\Mod}{Mod}

\newcommand{\tensor}{\otimes}

\renewcommand{\phi}{\varphi}

\renewcommand{\epsilon}{\varepsilon}

\newcommand{\setC}{\mathbb{C}}

\newcommand{\setN}{\mathbb{N}}

\newcommand{\setZ}{\mathbb{Z}}

\newcommand{\setQ}{\mathbb{Q}}

\usepackage[amsmath,thmmarks,hyperref]{ntheorem}

\newtheorem{thm}{Theorem}[section]

\newtheorem{prop}[thm]{Proposition}

\newtheorem{lemma}[thm]{Lemma}

\newtheorem{cor}[thm]{Corollary}

\newtheorem{defn}[thm]{Definition}

\newtheorem*{remark}{Remark}

\theoremstyle{nonumberplain}
\theoremheaderfont{%
\normalfont\bfseries}
\theorembodyfont{\normalfont}
\theoremsymbol{\ensuremath{\square}}
\theoremseparator{.}
\newtheorem{proof}{Proof}

\usepackage{enumitem}

\usepackage[final]{fixme}

\begin{document}

\title{Twisting functors for quantum group modules} \author{Dennis
  Hasselstrøm Pedersen} \date{}



\maketitle
\begin{abstract}
  We construct twisting functors for quantum group modules. First over
  the field $\setQ(v)$ but later over any
  $\setZ[v,v\inv]$-algebra. The main results in this paper are a
  rigerous definition of these functors, a proof that they satisfy
  braid relations and applications to Verma modules.
\end{abstract}




\section{Introduction}
Twisting functors were first introduced by S. Arkhipov (as a preprint
in 2001 and published in~\cite{Arkhipov}). H. Andersen quantized the
construction of twisting functors in~\cite{HHA-kvante}. Each twisting
functor $T_w$ is defined via a so called semi-regular bimodule
$S_v^w$. By the definition in~\cite{HHA-kvante} its right module
structure is not clear. Our first goal is to demonstrate that $S_v^w$
is in fact a bimodule. We verify this by constructing an explicit
isomorphism to an inductively defined right module. The calculations
are in fact rather complicated and involve several manipulations with
root vectors, see Section~\ref{sec:calc-with-root} below. At the same
time these calculations will be essential in~\cite{DHP1}
and~\cite{DHP2}.

Once we have established the definition of the twisting functors we
prove that they satisfy braid relations, see
Proposition~\ref{sec:twisting-functors-8}.  In the ordinary
(i.e. non-quantum) case the corresponding result was obtained by
O. Khomenko and V. Mazorchuck in~\cite{Marzorchuk-Khomenko}. Our
approach is similar but again the quantum case involves new
difficulties, see Section~\ref{sec:twisting-functors}. This section
also contains an explicit proof of the fact that, for the longest word
$w_0\in W$, the twisting functor $T_{w_0}$ takes a Verma module to its
dual, see Theorem~\ref{sec:twisting-functors-9},

The above results have several applications in the representation
theory of quantum group: They enable us to construct so called twisted
Verma modules and Jantzen filtrations of (twisted) Verma modules with
arbitrary (non-integral) weights and to derive the sum formula for
these. In turn this simplifies the linkage principle in quantum
category $\mathcal{O}_q$, q being a non-root of unity in an arbitrary
field.

\subsection{Acknowledgements}
I would like to thank my advisor Henning H. Andersen for great
supervision and many helpful comments and discussions. The authors
research was supported by the center of excellence grant 'Center for
Quantum Geometry of Moduli Spaces' from the Danish National Research
Foundation (DNRF95).

\subsection{Notation}
In this paper we work with a quantum group over a semisimple Lie
algebra $\mathfrak{g}$ defined as in~\cite{Jantzen}. Let $\Phi$
(resp. $\Phi^+$ and $\Phi^-$) denote the roots (resp. positive and
negative roots) and let $\Pi=\{\alpha_1,\dots,\alpha_n\}$ denote the
simple roots. The quantum group has generators
$\{E_\alpha,F_\alpha,K_\alpha|\alpha\in \Pi\}$ with relations as found
in~\cite{Jantzen}. Let $Q=\setZ\Phi$ denote the root lattice. Let
$(a_{ij})$ be the cartan matrix for $\mathfrak{g}$ and let
$(\cdot|\cdot)$ be the standard invariant bilinear form. Let
$\Lambda=\spa{\setZ}{\omega_1,\dots,\omega_n}\subset \mathfrak{h}^*$
be the integral lattice where $\omega_i\in \mathfrak{h}^*$ is the
fundamental weights defined by $(\omega_i|\alpha_j)=\delta_{ij}$. At
first we work with the quantum group $U_v(\mathfrak{g})$ defined over
$\setQ(v)$ but later we will specialize to an abitrary field and any
nonzero $q$ in the field. This is done by considering Lusztig's
$A$-form $U_A$ where $A=\setZ[v,v\inv]$, see
Section~\ref{sec:twist-funct-over}. For any $A$-algebra $R$;
$U_R=U_A\tensor_A R$. We will later need the automorphism $\omega$ of
$U_v$ and the antipode $S$ defined as in~\cite{Jantzen} along with the
definition of quantum numbers $[n]_\beta$ and quantum binomial
coefficients. We use the notation $E^{(r)}=\frac{E^r}{[r]!}$ and
similarly for $F$. The Weyl group $W$ is generated by the simple
reflections $s_i=s_{\alpha_i}$. For a $w\in W$ let $l(w)$ be the
length of $W$ i.e. the smallest amount of simple reflections such that
$w=s_{i_1}\cdots s_{i_{l(w)}}$. As usual we define for a weight
$\mu\in \Lambda$ the weight space $(U_v)_\mu:=\{u\in U_v|K_\alpha
u=v^{(\alpha|\mu)}u \text{ for all } \alpha\in \Pi\}$.  For a $\mu\in
Q$, $K_\mu$ is defined as follows: $K_\mu=\prod_{i=1}^n
K_{\alpha_i}^{a_i}$ if $\mu=\sum_{i=1}^n a_i \alpha_i$. There is a
braid group action on the quantum group $U_v$ usually denoted by
$T_{s_i}$ where $s_i$ is the reflection with respect to the simple
root $\alpha_i$.  In this paper we will reserve the $T$ for twisting
functors so we will call this braid group action $R$ instead. That is
we have automorphisms $R_{s_i}$ such that
\begin{align*}
  R_{s_i}E_{\alpha_i}=&-F_{\alpha_i}K_{\alpha_i}
  \\
  R_{s_i}E_{\alpha_j}=& \sum_{r+s=-a_{ij}} (-1)^s v_{\alpha_i}^{-s}
  E_{\alpha_i}^{(r)}E_{\alpha_j}E_{\alpha_i}^{(s)}, \text{ if } i\neq
  j
  \\
  R_{s_i}F_{\alpha_i}=&-K_{\alpha_i}\inv E_{\alpha_i}
  \\
  R_{s_i}F_{\alpha_j}=& \sum_{r+s=-a_{ij}} (-1)^s v_{\alpha_i}^{s}
  F_{\alpha_i}^{(s)}F_{\alpha_j}F_{\alpha_i}^{(r)}, \text{ if } i\neq
  j
  \\
  R_{s_i}K_{\mu}=&K_{s_i(\mu)}.
\end{align*}
Our definition of braid operators follows the definition
in~\cite{Jantzen}. Note that this definition differs slightly from the
original definition in~\cite{MR1066560}
(cf.~\cite[Warning~8.14]{Jantzen}).

The inverse to $R_{s_i}$ is given by
\begin{align*}
  R_{s_i}\inv E_{\alpha_i}=&-K_{\alpha_i} \inv F_{\alpha_i}
  \\
  R_{s_i}\inv E_{\alpha_j}=& \sum_{r+s=-a_{ij}} (-1)^s
  v_{\alpha_i}^{-s} E_{\alpha_i}^{(s)}E_{\alpha_j}E_{\alpha_i}^{(r)},
  \text{ if } i\neq j
  \\
  R_{s_i}\inv F_{\alpha_i}=&- E_{\alpha_i}K_{\alpha_i}
  \\
  R_{s_i}\inv F_{\alpha_j}=& \sum_{r+s=-a_{ij}} (-1)^s
  v_{\alpha_i}^{s} F_{\alpha_i}^{(r)}F_{\alpha_j}F_{\alpha_i}^{(s)},
  \text{ if } i\neq j
  \\
  R_{s_i}\inv K_{\mu}=&K_{s_i(\mu)}.
\end{align*}

For $w\in W$ with a reduced expression $s_{i_1}\cdots s_{i_r}$, $R_w$
is defined as $R_{s_{i_1}}\cdots R_{s_{i_r}}$. This is independent of
the reduced expression of $w$. An important property of the braid
operators is that if $\alpha_{i_1},\alpha_{i_2}\in\Pi$ and
$w(\alpha_{i_1})=\alpha_{i_2}$ then
$R_{w}(F_{\alpha_{i_1}})=F_{\alpha_{i_2}}$. These properties are
proved in Chapter~8 in~\cite{Jantzen}.

For a reduced expression $s_{i_1}\cdots s_{i_N}$ of $w_0$ we can make
an ordering of all the positive roots by defining
\begin{equation*}
  \beta_j:=s_{i_1}\cdots s_{i_{j-1}}(\alpha_{i_j}), \quad j=1,\dots,N
\end{equation*}
In this way we get $\{\beta_1,\dots,\beta_N\}=\Phi^+$. We could just
as well have used the opposite reduced expression $w_0=s_{i_N}\cdots
s_{i_1}$. In the following we will sometimes use the numbering
$s_{i_1}\cdots s_{i_N}$ and sometimes the numbering $s_{i_N}\cdots
s_{i_1}$.  Note that if $w=s_{i_1}\cdots s_{i_r}$ and we expand this
to a reduced expression $s_{i_{1}}\cdots s_{i_{r}}s_{i_{r+1}}\cdots
s_{i_N}$ we get $\{\beta_1,\dots,\beta_r\}=\Phi^+\cap w(\Phi^-)$. We
can define 'root vectors' $F_{\beta_j}, j=1,\dots,N$ by
\begin{equation*}
  F_{\beta_j}:=R_{s_{i_1}}\cdots R_{s_{i_{j-1}}}(F_{\alpha_{i_j}}).
\end{equation*}
Note that this definition depends on the chosen reduced
expression. For a different reduced expression we might get different
root vectors. As mentioned above if $\beta\in\Pi$ then the root vector
$F_\beta$ defined above is the same as the generator with the same
notation (cf. e.g~\cite[Proposition~8.20]{Jantzen}) so the notation is
not ambigious in this case. Let $w\in W$ and let $s_{i_r}\cdots
s_{i_1}$ be a reduced expression of $w$. Define $F_{\beta_j}$ by
choosing a reduced expression $s_{i_1}\cdots s_{i_r}s_{i_{r+1}}\cdots
s_{N}$ of $w_0$ starting with the reduced expression $s_{i_1}\cdots
s_{i_r}$ of $w\inv$. We define a subspace $U_v^-(w)$ of $U_v^-$ as
follows:
\begin{equation*}
  U_v^-(w):=\spa{\setQ(v)}{F_{\beta_1}^{a_1}\cdots F_{\beta_r}^{a_r}|a_j\in\setN}
\end{equation*}
where $F_{\beta_j}=R_{s_{i_1}}\cdots
R_{s_{i_{j-1}}}(F_{\alpha_{i_j}})$ as before.  The definition of
$U_v^-(w)$ seems to depend on the reduced expression of $w$. But the
subspace is independent of the chosen reduced expression. This is
shown in~\cite[Proposition~8.22]{Jantzen}. We will show below that
$U_v^-(w)$ is a subalgebra of $U_v^-$ and that
\begin{equation*}
  U_v^-(w)=\spa{\setQ(v)}{F_{\beta_r}^{a_r}\cdots F_{\beta_1}^{a_1}|a_j\in\setN}.
\end{equation*}

For a subalgebra $N\subset U_v$ we define $N^*=\bigoplus_{\mu}N_\mu^*$
(i.e. the graded dual) with the action given by $(uf)(x)=f(xu)$ for
$u\in U_v$, $f\in N^*$, $x\in N$.  We define 'the semiregular
bimodule' $S_v^w:=U_v\tensor_{U_v^-(w)}U_v^-(w)^*$. Proving that this
is a $U_v$-bimodule will be the first main result of this paper. We
will show that there exists a right module structure on $S_v^w$ such
that as a right module $S_v^w$ is isomorphic to
$U_v^-(w)^*\tensor_{U_v^-(w)}U_v$.

\section{Calculations with root vectors}
\label{sec:calc-with-root}

Let $A=\setZ[v,v\inv]$. Lusztigs $A$-form is defined to be the $A$
subalgebra of $U_v$ generated by the divided powers
$E_{\alpha_i}^{(n)}$ and $F_{\alpha_i}^{(n)}$ for $n\in \setN$ and
$K_i^{\pm 1}$.

We want to define
$U_A^-(w)=\operatorname{span}_A\left\{F_{\beta_1}^{(a_1)}\cdots
  F_{\beta_r}^{(a_r)}|a_i\in\setN\right\}$ where the $F_{\beta_i}$ are
defined from a reduced expression of $w$ like earlier. We have
$U_v^-(w_0)=U_v^-$ so we want a similar property over $A$:
$U_A^-(w_0)=U_A^-$ where $U_A^-$ is the $A$-subalgebra generated by
$\{F_{\alpha_i}^{(n)}|n\in\setN, i=1,\dots,n \}$. This is shown very
similar to the way it is shown for $U_v$ in~\cite{Jantzen}.

\begin{lemma}
  \label{sec:lusztigs-a-form-1}
  Assume $\mathfrak{g}$ does not contain any $G_2$ components:
  \begin{enumerate}
  \item The subspace
    $U_A(w):=\operatorname{span}_A\left\{F_{\beta_1}^{(a_1)}\cdots
      F_{\beta_r}^{(a_r)}|a_i\in\setN\right\}$ depends only on $w$,
    not on the reduced expression chosen for $w$.
  \item Let $\alpha$ and $\beta$ be two distinct simple roots. If $w$
    is the longest element in the subgroup of $W$ generated by
    $s_\alpha$ and $s_b$ then the span defined as before is the
    subalgebra of $U_A$ generated by $F_\alpha^{(a)}$ and
    $F_\beta^{(b)}$, $a,b\in \setN$.
  \end{enumerate}
\end{lemma}
\begin{proof}
  Claim 2. is shown on a case by case basis. We will show first that
  the second claim implies the first.

  We show this by induction on $l(w)$. If $l(w)\leq 1$ then there is
  only one reduced expression of $w$ and there is nothing to
  show. Assume $l(w)>1$ and that $w$ has two reduced expressions
  $w=s_{\alpha_1}s_{\alpha_2}\cdots s_{\alpha_r}$ and
  $w=s_{\gamma_1}s_{\gamma_2}\cdots s_{\gamma_r}$. We can assume that
  we can get from one of the reduced expression to the other by an
  elementary braid move ($s_\alpha s_\beta\cdots = s_\beta s_\alpha
  \cdots$). Set $\alpha=\alpha_1$ and $\gamma=\gamma_1$.

  If $\alpha=\gamma$, set $w'=s_\alpha w$. Then the subspace spanned
  by the elements as in the lemma is for both expressions equal to:
  \begin{equation}
    \label{eq:reducedexpr}
    \left( \sum_{a\geq 0} F_\alpha^{(a)}\right) \cdot R_{s_\alpha}(U_A^-(w'))
  \end{equation}
  If $\alpha\neq \gamma$ then the elementary move must take place at
  the beginning of the reduced expression for both reduced
  expressions. Let $w''$ be the longest element generated by
  $s_\alpha$ and $s_\gamma$ then we must have $w=w''w'$ for some $w'$
  with $l(w'')+l(w')=l(w)$ and the reduced expression for $w'$ in both
  reduced expressions are equal whereas the reduced expressions for
  $w''$ are the two possible combinations for the two different
  reduced expressions.  So the span of the products is given by
  $U_A^-(w')R_{w''}(U_A^-(w''))$ which is independent of the reduced
  expression by the second claim.
  
  We turn to the proof of the second claim: First assume we are in the
  simply laced case. Then $w=s_\alpha s_\beta s_\alpha = s_\beta
  s_\alpha s_\beta$. Lets work with the reduced expression $s_\alpha
  s_\beta s_\alpha$. The other situation is symmetric by changing the
  role of $\alpha$ and $\beta$. We want to show that
  \begin{equation}
    \label{eq:2}
    B:=\left< F_\alpha^{(n_1)},F_\beta^{(n_2)}|n_1,n_2\in \setN\right>_A = \operatorname{span}_A\left\{F_\alpha^{(a_1)}F_{\alpha+\beta}^{(a_2)}F_\beta^{(a_3)}|a_i\in\setN \right\}=: V
  \end{equation}
  where
  $F_{\alpha+\beta}^{(a)}=R_\alpha(F_\beta^{(a)})$. By~\cite{MR1066560}
  section 5 we have that $F_{\alpha+\beta}^{(a)}\in U_A^-$ for all
  $a\in\setN$ and we see that
  \begin{equation*}
    F_\beta^{(k)}F_\alpha^{(k')}=\sum_{t,s\geq 0} (-1)^s v^{-tr-s}F_\alpha^{(r)}F_{\alpha+\beta}^{(s)}F_\beta^{(t)}
  \end{equation*}
  where the restrictions on the sum is $s+t=k'$ and $s+t=k$. Lusztig
  calculates for the $E_\alpha$'s but just use the anti-automorphism
  $\Omega$ (defined in Section~1 of~\cite{MR1066560}) on the results
  to get the corresponding formulas for the $F$'s. Also we get the
  $(-1)^s$ from the fact that (using the notation of~\cite{MR1066560})
  $E_{12}=-R_{\alpha_2}(E_{\alpha_1})$ because of the difference in
  the definition of the braid operators. Since
  $F_{\alpha+\beta}^{(a)}\in U_A^-$ we have that $V\subset B$. If we
  show that $V$ is invariant by multiplication from the left with
  $F_\alpha^{(a)}$ and $F_{\beta}^{(a)}$ for all $a\in\setN$ then we
  must have $B\subset V$. For $F_\alpha^{(a)}$ this is clear. For
  $F_\beta^{(k)}$, $k\in\setN$ we use the formula above:
  \begin{align*}
    F_\beta^{(k)}F_\alpha^{(a_1)}F_{\alpha+\beta}^{(a_2)}F_\beta^{(a_3)}
    =& \sum_{t,s\geq 0} (-1)^s v^{-d(tr+s)}
    F_\alpha^{(r)}F_{\alpha+\beta}^{(s)}F_\beta^{(t)}
    F_{\alpha+\beta}^{(a_2)}F_\beta^{(a_3)}
    \\
    =& \sum_{t,s\geq 0} (-1)^s v^{-d(tr+s)+dta_2}
    F_\alpha^{(r)}F_{\alpha+\beta}^{(s)}
    F_{\alpha+\beta}^{(a_2)}F_\beta^{(t)}F_\beta^{(a_3)}
    \\
    =& \sum_{t,s\geq 0} (-1)^s v^{-d(tr+s)+dta_2} {s+a_2 \brack
      s}{t+a_3\brack t} F_\alpha^{(r)}
    F_{\alpha+\beta}^{(s+a_2)}F_\beta^{(t+a_3)}.
  \end{align*}
  We see that $F_\beta^{(k)}V\subset V$ so $V=B$.

  In the non simply laced case we have to use the formulas
  in~\cite{MR1066560} section 5.3 (d)-(i) but the idea of the proof is
  the same. If there were similar formulas for the $G_2$ case it would
  be possible to show the same here. I do not know if similar formulas
  can be found in this case. The important part is just that if you
  'v-commute' two of the 'root vectors' $F_{\beta_i}^{(k)}$ and
  $F_{\beta_j}^{(k')}$ you get something that is still in $U_A$.
\end{proof}


\begin{lemma}
  \label{sec:lusztigs-a-form}
  \begin{equation*}
    U_A^-(w_0)=U_A^-
  \end{equation*}
\end{lemma}
\begin{proof}
  It is clear that $U_A^-(w_0)\subset U_A^-$. We want to show that
  $F_\alpha^{(k)}U_A^-(w_0)\subset U_A(w_0)$ for all $\alpha\in \Pi$.

  $U_A^-(w_0)$ is independent of the chosen reduced expression so we
  can choose a reduced expression for $w_0$ such that $s_\alpha$ is
  the last factor. Then the first root vector $F_{\beta_1}$ is equal
  to $F_\alpha$. Then it is clear that
  $F_\alpha^{(k)}U_A^-(w_0)\subset U_A^-(w_0)$. Since this was for an
  abitrary simple root $\alpha$ the proof is finished. (This argument
  is sketched in the appendix of~\cite{MR1066560}.)
\end{proof}

\begin{cor}
  We get a basis of $U_A^-$ by the products of the form
  $F_{\beta_1}^{(a_1)}\cdots F_{\beta_N}^{(a_N)}$ where
  $a_1,\dots,a_N\in \setN$.
\end{cor}

\begin{cor}
  \label{cor:1}
  $U_A^-(w)=U_v^-(w)\cap U_A$.
\end{cor}
\begin{proof}
  Assume the length of $w$ is $r$ and define for $k=(k_1,\dots,k_r)\in \setN^r$
  \begin{equation*}
    F^{(k)}=F_{\beta_1}^{(k_1)}\cdots F_{\beta_r}^{(k_r)}.
  \end{equation*}
  It is clear that $U_A^-(w)\subseteq U_v^-(w)\cap U_A$. Assume $x\in
  U_v^-(w)\cap U_A$. Since $x\in U_v^-(w)$ we have constants $c_k\in
  \setQ(v)$, $k\in \setN^r$ such that
  \begin{equation*}
    x=\sum_{k\in\setN^r} c_k F^{(k)}.
  \end{equation*}
  Assume the length of $w_0$ is $N$ and denote for $n\in \setN^N$,
  $F^{(n)}$ like above for $w$.  $U_v^-(w)\cap U_A\subseteq
  U_v^-(w_0)\cap U_A=U_A^-(w_0)$ ($U_A^-(w_0)\subset U_v^-(w_0)\cap
  U_A$ clearly and $U_A^-(w_0)$ is invariant under multiplication by
  $U_A^-$.) so there exists $b_{n}\in A$, $n\in \setN^N$ such that
  \begin{equation*}
    x=\sum_{k\in\setN^N} b_k F^{(k)}.
  \end{equation*}
  But then we have two expressions of $x$ in $U_v^-(w)$ expressed as a
  linear combination of basis elements. So we must have that the
  multindieces $b_k$ are zero on coordinates $\geq$ $r$ and that all
  the $c_k$ are actually in $A$. This proves the corollary.
\end{proof}

\begin{defn}
  \label{sec:twisting-functors-2}
  Let $x\in (U_v)_\mu$ and $y\in (U_v)_\gamma$ then
  \[ [x,y]_v:=xy-v^{-(\mu|\gamma)}yx.
  \]
\end{defn}

\begin{prop}
  \label{derivative}
  For $x_1\in (U_v)_{\mu_1}$, $x_2\in (U_v)_{\mu_2}$ and $y\in
  (U_v)_\gamma$ we have
  \begin{equation*}
    [x_1x_2,y]_v=x_1[x_2,y]_v+v^{-(\gamma|\mu_2)}[x_1,y]_vx_2
  \end{equation*}
  and
  \begin{equation*}
    [y,x_1x_2]_v=v^{-(\gamma|\mu_1)}x_1[y,x_2]_v+[y,x_1]_vx_2.
  \end{equation*}
\end{prop}
\begin{proof}
  Direct calculation.
\end{proof}

We have the following which corresponds to the Jacobi identity. Note
that setting $v=1$ recovers the usual Jacobi identity for the
commutator.
\begin{prop}
  \label{prop:3}
  for $x \in (U_v)_\mu$, $y\in (U_v)_\nu$ and $z\in (U_v)_\gamma$ we
  have
  \begin{equation*}
    [[x,y]_v,z]_v = [x,[y,z]_v]_v - v^{-(\mu|\nu)}[y,[x,z]_v]_v + v^{-(\nu|\mu+\gamma)}\left(v^{(\nu|\mu)}-v^{-(\nu|\mu)}\right) [x,z]_v y
  \end{equation*}
\end{prop}
\begin{proof}
  Direct calculation.
\end{proof}

For use in the theorem below define:
\begin{defn}
  Let $A=\setZ[v,v\inv]$ and let $A'$ be the localization of $A$ in
  $[2]$ (and/or $[3]$) if the Lie algebra contains any $B_n,C_n$ or
  $F_4$ part (resp. any $G_2$ part). Let $w\in W$ have a reduced
  expression $s_{i_r}\cdots s_{i_1}$. Define $\beta_j$ and
  $F_{\beta_j}$, $i=1,\cdots, r$ as above: $\beta_j=s_{i_1}\cdots
  s_{i_{j-1}}(\alpha_{i_j})$ and $F_{\beta_{j}}=R_{s_{i_1}}\cdots
  R_{s_{i_{j-1}}}(F_{\alpha_{i_j}})$. We define
  \begin{equation*}
    U_{A'}^-(w)=\spa{A'}{F_{\beta_1}^{a_1}\cdots F_{\beta_r}^{a_r}|a_1,\dots,a_r\in\setN}
  \end{equation*}
\end{defn}
This subspace is independent of the reduced expression for $w$. This
can be proved in the same way as Lemma~\ref{sec:lusztigs-a-form-1}
using the rank~$2$ calculations done in~\cite{MR1066560}.

The main tool that will be used in this project is the following
theorem from~\cite[thm 9.3]{DP} originally
from~\cite[Proposition~5.5.2]{Levendorski-Soibelman}:
\begin{thm}
  \label{thm:DP}
  Let $F_{\beta_j}$ and $F_{\beta_i}$ be defined as above. Let
  $i<j$. Let $A=\setZ[v,v\inv]$ and let $A'$ be the localization of
  $A$ in $[2]$ (and/or $[3]$) if the Lie algebra contains any
  $B_n,C_n$ or $F_4$ part (resp. any $G_2$ part). Then
  \begin{equation*}
    [F_{\beta_j},F_{\beta_i}]_v=F_{\beta_j}F_{\beta_i}-v^{-(\beta_i|\beta_j)}F_{\beta_i}F_{\beta_j}\in \spa{A'}{F_{\beta_{i+1}}^{a_{i+1}}\cdots F_{\beta_{j-1}}^{a_{j-1}}}
  \end{equation*}
\end{thm}
\begin{proof}
  We shall provide the details of the proof sketched in~\cite{DP}.
  The rank 2 case is handled in~\cite{MR1066560}. Note that
  in~\cite{MR1066560} we see that when $\mu=2$ (in his notation) we
  get second divided powers and when $\mu=3$ we get third divided
  powers. This is one reason why we need to be able to divide by $[2]$
  and $[3]$.

  So we assume the rank $2$ case is proven. In particular we can
  assume there is no $G_2$ component. Let $k\in\setN$, $k<j$. Then
  $[F_{\beta_j},F_{\beta_k}]=R_{s_{i_1}}\cdots
  R_{s_{i_{k-1}}}[R_{s_{i_k}}\cdots
  R_{s_{i_{j-1}}}(F_{\alpha_{i_j}}),F_{{\alpha_{i_k}}}]_v$ so we can
  assume in the above that $i=1$. We can then assume that $j>2$
  because otherwise we would be in the rank $2$ case. We will show by
  induction over $l\in \setN$ that
  \begin{equation*}
    [F_{\beta_t},F_{\beta_1}]_v=F_{\beta_t}F_{\beta_1}-v^{-(\beta_1|\beta_t)}F_{\beta_1}F_{\beta_t}\in \spa{A'}{F_{\beta_{2}}^{a_{2}} \cdots F_{\beta_{t-1}}^{a_{t-1}}}
  \end{equation*}
  for all $1< t \leq l$.  The induction start $l=2$ is the rank $2$
  case. Assume the induction hypothesis that
  \begin{equation*}
    [F_{\beta_t},F_{\beta_1}]_v=F_{\beta_t}F_{\beta_1}-v^{-(\beta_1|\beta_t)}F_{\beta_1}F_{\beta_t}\in \spa{A'}{F_{\beta_{2}}^{a_{2}} \cdots F_{\beta_{t-1}}^{a_{t-1}}}
  \end{equation*}
  for $t\leq l$. We need to prove the result for $l+1$.  We have
  $\beta_{l+1}=s_{i_1}\cdots s_{i_{l}}(\alpha_{i_{l+1}})$. Now define
  $i=i_l$ and $j=i_{l+1}$. Set $w=s_{i_1}\cdots s_{i_{l-1}}$. So
  $\beta_{l+1}=ws_i(\alpha_j)$ and
  $F_{\beta_{l+1}}=R_{w}R_{s_{i}}(F_{\alpha_j})$. Define
  $\alpha=\alpha_{i_1}$. We need to show that
  \begin{equation*}
    [R_{w}R_{s_i}(F_{\alpha_j}),F_\alpha]_v \in \spa{A'}{F_{\beta_{2}}^{a_{2}}\cdots F_{\beta_{l}}^{a_l}}.
  \end{equation*}
  We divide into cases:

  Case 1) $(\alpha_i|\alpha_j)=0$: In this case
  $R_{w}R_{s_i}(F_{\alpha_j})=R_w(F_{\alpha_j})$. Since
  $s_is_j=s_js_i$ there is a reduced expression for $w_0$ starting
  with $s_{i_1}\cdots s_{l-1} s_js_i$. So the induction hypothesis
  gives us that $[R_w(F_{\alpha_j}),F_{\alpha_1}]_v$ can be expressed
  by linear combinations of ordered monomials involving only
  $F_{\beta_2}\cdots F_{\beta_{l-1}}$.

  Case 2) $(\alpha_i|\alpha_j)=-1$ and $l(ws_j)>l(w)$: In this case
  $ws_is_j(\alpha_i)=w(\alpha_j)>0$ so there is a reduced expression
  for $w_0$ starting with $s_{i_1} \cdots s_{i_{l-1}}s_is_js_i=s_{i_1}
  \cdots s_{i_{l-1}}s_js_is_j$. So we have by induction that $[R_w
  (F_{\alpha_j}), F_\alpha]_v$ is a linear combination of ordered
  monimials only involving $F_{\beta_2}\cdots F_{\beta_{l-1}}$.

  Observe that we have
  \begin{align*}
    F_{\beta_{l+1}} =& R_w R_{s_i}(F_{\alpha_j})
    \\
    =& R_w( F_{\alpha_j}F_{\alpha_i} - v F_{\alpha_i} F_{\alpha_j})
    \\
    =& R_w(F_{\alpha_j})F_{\beta_l}- v F_{\beta_l}R_w(F_{\alpha_j})
    \\
    =& [R_w(F_{\alpha_j}),F_{\beta_l}]_v
  \end{align*}
  so by Proposition~\ref{prop:3} we get
  \begin{align*}
    [F_{\beta_{l+1}},F_\alpha]_v =&
    [[R_w(F_{\alpha_j}),F_{\beta_l}]_v,F_\alpha]_v
    \\
    =& [R_w(F_{\alpha_j}),[F_{\beta_l},F_\alpha]_v]_v -
    v^{-(w(\alpha_j)|\beta_l)}
    [F_{\beta_l},[R_w(F_{\alpha_j}),F_\alpha]_v]_v
    \\
    &+ v^{-(\beta_l|\alpha+w(\alpha_j))}\left( v\inv - v\right)
    [R_w(F_{\alpha_j}),F_\alpha]_v F_{\beta_l}.
  \end{align*}
  By induction (and Proposition~\ref{derivative})
  $[R_w(F_{\alpha_j}),[F_{\beta_l},F_\alpha]_v]_v$ and
  $[F_{\beta_l},[R_w(F_{\alpha_j}),F_\alpha]_v]_v$ are linear
  combinations of ordered monomials containing only
  $F_{\beta_2},\dots,F_{\beta_{l-1}}$ so we have proved this case.

  Case 3) $(\alpha_i|\alpha_j)=-1$ and $l(ws_j)<l(w)$: In this case
  write $u=ws_j$. We claim $l(us_i)>l(u)$. Assume $l(us_i)<l(u)$ then
  \begin{equation*}
    l(w)+2=l(ws_i s_j)=l(us_js_is_j)=l(us_is_js_i)<l(u)+2=l(w)+1
  \end{equation*}
  A contradiction.  So there is a reduced expression of $w_0$ starting
  with $us_i$. We have
  $F_{\beta_{l+1}}=R_wR_{s_i}(F_{\alpha_j})=R_u(F_{\alpha_i})$ so we
  get
  \begin{equation*}
    [F_{\beta_{l+1}},F_\alpha]_v=[R_u(F_{\alpha_i}),F_{\alpha}]_v
  \end{equation*}
  Now we claim that either $u\inv(\alpha)=\alpha_j$ or
  $u\inv(\alpha)<0$: Indeed $w\inv(\alpha)<0$ so $u\inv(\alpha)$ is
  $<0$ unless $w\inv(\alpha)=-\alpha_j$ in which case we get
  $u\inv(\alpha)=s_jw\inv(\alpha)=s_j(-\alpha_j)=\alpha_j$.  If
  $\alpha=u(\alpha_j)$ we get
  \begin{equation*}
    [R_u(F_{\alpha_i}),F_\alpha]_v= R_u([F_{\alpha_i},F_{\alpha_j}]_v)=R_u(R_{s_j}(F_{\alpha_i}))=R_w(F_{\alpha_i})=F_{\beta_l}
  \end{equation*}
  In the other case we know from induction that
  \begin{equation*}
    [R_u(F_{\alpha_i}),F_\alpha]_v\in U_{A'}^-(u\inv)
  \end{equation*}
  Now $U_{A'}^-(u\inv)\subset U_{A'}^-(s_ju\inv) = U_{A'}^-(w\inv)$ so
  we get that $[R_u(F_{\alpha_i}),F_\alpha]_v$ can be expressed as a
  linear combination of monomials involving $F_\alpha=F_{\beta_1}$ and
  the terms $F_{\beta_2}\cdots F_{\beta_{l-1}}$. Assume that a
  monomial of the form $F_{\alpha}^m F_{\beta_2}^{a_2}\cdots
  F_{\beta_{l-1}}^{a_{l-1}} $ appears with nonzero coefficient. The
  weights of the left and right hand side must agree so we have
  $ws_i(\alpha_j)+\alpha=\sum_{k=2}^{l-1} a_k \beta_k + m\alpha$ or
  \begin{equation*}
    ws_i(\alpha_j)=\sum_{k=2}^{l-1} a_k \beta_k + (m-1)\alpha
  \end{equation*}
  Since $w\inv(\beta_k)<0$ for $k=1,2,\dots l-1$ (and
  $\alpha=\beta_1$) we get
  \begin{equation*}
    \alpha_i+\alpha_j=w\inv ws_i(\alpha_j)=\sum_{k=2}^{l-1} a_k w\inv (\beta_k) + (m-1)w\inv (\alpha)<0.
  \end{equation*}
  Which is a contradiction.

  Case 4) $\left<\alpha_j,\alpha_i^\vee\right>=-1$,
  $(\alpha_i|\alpha_j)=-2$ and $l(ws_j)>l(w)$: Here we get
  \begin{equation*}
    F_{\beta_{l+1}}=R_wR_{s_i}(F_{\alpha_j})=R_w(F_{\alpha_j}F_{\alpha_i}-v^2 F_{\alpha_i}F_{\alpha_j})=R_w(F_{\alpha_j})F_{\beta_l}-v^2F_{\beta_l}R_w(\alpha_j)=[R_w(F_{\alpha_j}),F_{\beta_l}]_v
  \end{equation*}
  From here the proof goes exactly as in case 2.

  Case 5) $\left<\alpha_j,\alpha_i^\vee\right>=-2$, and
  $l(ws_j)>l(w)$: First of all since $l(ws_j)>l(w)$ we can deduce that
  $l(ws_is_js_is_j)=l(w)+4$: We have
  $-\beta_{l+1}+2ws_is_j(\alpha_i)=ws_is_js_i(\alpha_j)=w(\alpha_j)>0$
  showing that we must have $ws_is_j(\alpha_i)>0$.

  We have
  \begin{equation*}
    F_{\beta_{l+1}}=R_wR_{s_i}(F_{\alpha_j})=R_w(F_{\alpha_i}F_{\alpha_j}^{(2)}-vF_{\alpha_j}F_{\alpha_i}F_{\alpha_j}+v^2 F_{\alpha_j}^{(2)}F_{\alpha_i})
  \end{equation*}
  We claim that we have
  \begin{equation*}
    R_{s_i}(F_{\alpha_j})=\frac{1}{[2]}\left(R_{s_i}R_{s_j}(F_{\alpha_i})F_{\alpha_i}-F_{\alpha_i}R_{s_i}R_{s_j}(F_{\alpha_i})\right)
  \end{equation*}
  This is shown by a direct calculation. First note that
  \begin{equation*}
    R_{s_i}R_{s_j}(F_{\alpha_i})=R_{s_j}\inv R_{s_j}R_{s_i}R_{s_j}(F_{\alpha_i})=R_{s_j}\inv(F_{\alpha_i})=F_{\alpha_j}F_{\alpha_i}-v^2 F_{\alpha_i}F_{\alpha_j}
  \end{equation*}
  So
  \begin{align*}
    R_{s_i}R_{s_j}(F_{\alpha_i})F_{\alpha_i}-F_{\alpha_i}R_{s_i}R_{s_j}(F_{\alpha_i})=&
    F_{\alpha_j}F_{\alpha_i}^2-v^2
    F_{\alpha_i}F_{\alpha_j}F_{\alpha_i}-F_{\alpha_i}F_{\alpha_j}F_{\alpha_i}+v^2
    F_{\alpha_i}^2F_{\alpha_j}
    \\
    =& F_{\alpha_j}F_{\alpha_i}^2- v [2]
    F_{\alpha_i}F_{\alpha_j}F_{\alpha_i}+v^2
    F_{\alpha_i}^2F_{\alpha_j}
    \\
    =& [2]R_{s_i}(F_{\alpha_i}).
  \end{align*}
  Therefore
  \begin{align*}
    F_{\beta_{l+1}}=&\frac{1}{[2]}\left(R_wR_{s_i}R_{s_j}(F_{\alpha_i})F_{\beta_l}-F_{\beta_l}R_wR_{s_i}R_{s_j}(F_{\alpha_i})
    \right)
    \\
    =& \frac{1}{[2]} [R_wR_{s_i}R_{s_j}(F_{\alpha_i}),F_{\beta_l}]_v
    \\
    =& \frac{1}{[2]} [[R_w(F_{\alpha_j}),F_{\beta_l}]_v,F_{\beta_l}]_v
  \end{align*}
  
  By Proposition~\ref{prop:3} and the above we get
  \begin{align*}
    [R_w R_{s_i}R_{s_j}(F_{\alpha_i}),F_\alpha]_v =&
    [[R_w(F_{\alpha_j}),F_{\beta_l}]_v,F_\alpha]_v
    \\
    =& [R_w(F_{\alpha_j}),[F_{\beta_l},F_\alpha]_v]_v - v^2
    [F_{\beta_l},[R_w(F_{\alpha_j}),F_\alpha]_v]_v
    \\
    +& v^{2-(\alpha|\beta_l)}\left( v^{-2}-v^2 \right)
    [R_w(F_{\alpha_j}),F_{\alpha}]_v F_{\beta_l}
  \end{align*}
  which by induction is a linear combination of ordered monomials
  involving only $F_{\beta_2},\dots, F_{\beta_l}$. Using
  Proposition~\ref{prop:3} again we get
  \begin{align*}
    [2][F_{\beta_{l+1}},F_\alpha]_v =&
    [[R_wR_{s_i}R_{s_j}(F_{\alpha_i}),F_{\beta_l}]_v,F_\alpha]_v
    \\
    =& [R_wR_{s_i}R_{s_j}(F_{\alpha_i}),[F_{\beta_l},F_{\alpha}]_v]_v
    - [F_{\beta_l},[R_wR_{s_i}R_{s_j}(F_{\alpha_i}),F_\alpha]_v]_v
  \end{align*}
  which by induction and the above is a linear combination of ordered
  monomials involving only $F_{\beta_2},\dots, F_{\beta_l}$.


  Case 6) $(\alpha_i|\alpha_j)=-2$, $l(ws_j)<l(w)$ and
  $l(ws_js_i)<l(ws_j)$: Set $u=ws_js_i$. We claim
  $l(us_i)=l(us_j)>l(u)$. Indeed suppose the contrary then
  $l(w)+2=l(ws_is_j)=l(us_is_js_is_j)<l(u)+4=l(w)+2$. We reason like
  in case 3): We have
  $F_{\beta_{l+1}}=R_wR_{s_i}(F_{\alpha_j})=R_{u}R_{s_i}R_{s_j}R_{s_i}(F_{\alpha_j})=R_u(F_{\alpha_j})$. Now
  either $u\inv(\alpha)=\alpha_i$, $u\inv(\alpha)=s_i(\alpha_j)$ or
  $u\inv(\alpha)<0$.  If $u\inv (\alpha)<0$ we get by induction that
  $[F_{\alpha},R_u(F_{\alpha_j})]_v$ is in $U_{A'}^-(u\inv)\subset
  U_{A'}^-(w\inv)$ and by essentially the same weight argument as in
  case 3) we are done.

  If $\alpha=u(\alpha_i)$ then
  \begin{align*}
    [R_u(F_{\alpha_j}),F_\alpha]_v=&[R_u(F_{\alpha_j}),R_u(F_{\alpha_i})]_v
    \\
    =& R_u(F_{\alpha_j}F_{\alpha_i}-v^2F_{\alpha_i}F_{\alpha_j})
    \\
    =&
    \begin{cases}
      R_uR_{s_i}(F_{\alpha_j}) &\text{ if }
      \left<\alpha_j,\alpha_i^\vee\right>=-1
      \\
      R_uR_{s_i}R_{s_j}(F_{\alpha_i}) &\text{ if }
      \left<\alpha_j,\alpha_i^\vee\right>=-2
    \end{cases}
  \end{align*}
  So $[F_\alpha,R_u(F_{\alpha_j})]_v\in
  U_{A'}^-(s_is_js_iu\inv)=U_{A'}^-(s_iw\inv)$.  Assume we have a
  monomial of the form $F_\alpha^{m}F_{\beta_2}^{a_2}\cdots
  F_{\beta_l}^{a_l} $ with $m$ nonzero in the expression of
  $[R_u(F_{\alpha_j}),F_\alpha]_v$. Then
  \begin{equation*}
    ws_i(\alpha_j)=\sum_{k=2}^l a_k\beta_k + (m-1)\alpha
  \end{equation*}
  and we get
  \begin{equation*}
    \alpha_j=\sum_{k=2}^l a_k s_iw\inv (\beta_k) + (m-1)s_iw\inv (\alpha)<0.
  \end{equation*}
  A contradiction.

  If $\alpha=us_i(\alpha_j)$ then
  \begin{align*}
    [R_u(F_{\alpha_j}),F_\alpha]_v=&R_u[F_{\alpha_j},R_{s_i}(F_{\alpha_j})]_v
    \\
    =&
    R_u(F_{\alpha_j}R_{s_i}(F_{\alpha_j})-v^{-2}R_{s_i}(F_{\alpha_i})F_{\alpha_j})
    \\
    =&
    R_u(R_{s_i}R_{s_j}R_{s_i}(F_{\alpha_j})R_{s_i}(F_{\alpha_j})-v^{-2}R_{s_i}(F_{\alpha_j})R_{s_i}R_{s_j}R_{s_i}(F_{\alpha_j}))
  \end{align*}
  Which is in $U_{A'}^-(s_is_js_iu\inv)=U_{A'}^-(s_iw\inv)$ by the
  rank 2 case. By the same weight argument as above we are done.

  Case 7) $(\alpha_i|\alpha_j)=-2$, $l(ws_j)<l(w)$ and
  $l(ws_js_i)>l(ws_j)$: Set $u=ws_j$. Like in case 3) we get that
  either $u\inv (\alpha)=\alpha_j$ or $u\inv(\alpha)<0$. If
  $\alpha=u(\alpha_j)$:
  \begin{equation*}
    [F_{\beta_{l+1}},F_\alpha]_v=R_u[R_{s_j}R_{s_i}(F_{\alpha_j}),F_{\alpha_j}]_v\in U_{A'}^-(s_is_ju\inv) = U_{A'}^-(s_iw\inv)
  \end{equation*}
  And by a weight argument as above we are done.

  If $u\inv(\alpha)<0$ then $\alpha=\beta_i'$ for some
  $i\in\{1,\dots,l-2\}$ where the $\beta_i'$s are defined as above but
  using a reduced expression of $u$. Set $\beta_{l-1}'=u(\alpha_j)$,
  $\beta_l'=us_j(\alpha_i)$ and
  $\beta_{l+1}'=us_js_i(\alpha_j)=ws_i(\alpha_j)=\beta_{l+1}$. Then
  \begin{equation*}
    [F_{\beta_{l+1}},F_\alpha]_v=[F_{\beta_{l+1}'},F_{\beta_i'}]_v\in U_{A'}^-(s_is_ju\inv)=U_{A'}^-(s_iw\inv)
  \end{equation*}
  by induction and by a weight argument as above we are done.
\end{proof}

\begin{lemma}
  \label{algebra}
  Let $w_0=s_{i_1}\cdots s_{i_N}$ and let
  $F_{\beta_j}=R_{s_{i_1}}\cdots R_{s_{i_{j-1}}}(F_{\alpha_{i_j}})$
  let $l,r\in \{1,\dots,N\}$ with $l\leq r$. Then
  \[\spa{\setQ(v)}{F_{\beta_r}^{a_r} \cdots F_{\beta_l}^{a_l}|a_j
    \in\setN}=\spa{\setQ(v)}{F_{\beta_l}^{a_l}
    \cdots F_{\beta_r}^{a_r}|a_j \in\setN}
  \]
  and the subspace is invariant under multiplication from the left by
  $F_{\beta_i}$, $i=l,\dots, r$.
\end{lemma}
\begin{proof}
  If $r-l=0$ the lemma obviously holds. Assume $r-l>0$. For $k\in
  \setN^{r-l}$, $k=(k_l,\dots,k_r)$ let $F^k=F_{\beta_l}^{k_l}\cdots
  F_{\beta_r}^{k_r}$. We will prove the statement that $F^k\in
  \spa{\setQ(v)}{F_{\beta_r}^{a_r} \cdots F_{\beta_l}^{a_l}|a_j
    \in\setN}$ by induction over $k_l+\cdots +k_r$.  If $k=0$ the
  statement holds. We have
  \[F^k=F_{\beta_j}F_{\beta_j}^{k_j-1}F_{\beta_{j+1}}^{k_{j+1}}\cdots
  F_{\beta_r}^{k_r}.
  \]
  By induction $F_{\beta_j}^{k_j-1}F_{\beta_{j+1}}^{k_{j+1}}\cdots
  F_{\beta_r}^{k_r}\in \spa{\setQ(v)}{F_{\beta_r}^{a_r} \cdots
    F_{\beta_l}^{a_l}|a_j \in\setN}$ so if we show that
  $F_{\beta_j}F_{\beta_r}^{b_r} \cdots
  F_{\beta_l}^{b_l}\in\spa{\setQ(v)}{F_{\beta_r}^{a_r} \cdots
    F_{\beta_l}^{a_l}|a_i\in\setN}$ for all $b_i$, $i=l,\dots r$ then
  we have shown the first inclusion.

  We use downwards induction on $j$ and induction on $b_1+\cdots
  +b_r$. If $j=r$ then this is obviously true. If $j<r$ we use
  theorem~\ref{thm:DP} to conclude that
  \[F_{\beta_r}F_{\beta_j}-v^{-(\beta_r|\beta_j)}F_{\beta_j}F_{\beta_r}\in
  \spa{\setQ(v)}{F_{\beta_{r-1}}^{a_{r-1}}\cdots
    F_{\beta_{j+1}}^{a_{j+1}}|a_i\in\setN}
  \]
  If $b_r=0$ the induction over $j$ finishes the claim. We get now if
  $b_r\neq 0$
  \begin{align*}
    F_{\beta_j}F_{\beta_r}^{b_r} \cdots F_{\beta_l}^{b_l}=
    v^{(\beta_r|\beta_j)}\left(F_{\beta_r}F_{\beta_j}F_{\beta_r}^{b_r-1}
      \cdots F_{\beta_l}^{b_l}+\Sigma F_{\beta_r}^{b_r-1} \cdots
      F_{\beta_l}^{b_l} \right)
  \end{align*}
  where $\Sigma\in \spa{\setQ(v)}{F_{\beta_{r-1}}^{a_{r-1}}\cdots
    F_{\beta_{j+1}}^{a_{j+1}}|a_i\in\setN}$. By the induction on
  $b_r+\dots+b_l$ $F_{\beta_j}F_{\beta_r}^{b_r-1}\cdots
  F_{\beta_l}^{b_l}\in \spa{\setQ(v)}{F_{\beta_{r}}^{a_r}\cdots
    F_{\beta_l}^{a_l}|a_i\in\setN}$ and the induction on $j$ ensures
  that $\Sigma F_{\beta_r}^{b_r-1}\cdots F_{\beta_l}^{b_l}\in
  \spa{\setQ(v)}{F_{\beta_{r}}^{a_r}\cdots
    F_{\beta_l}^{a_l}|a_i\in\setN}$ since $\Sigma$ contains only
  elements generated by $F_{\beta_{r-1}}\cdots F_{\beta_l}$.
  
  We have now shown that
  \[\spa{\setQ(v)}{F_{\beta_l}^{a_l} \cdots F_{\beta_r}^{a_r}|a_j
    \in\setN}\subset
  \spa{\setQ(v)}{F_{\beta_r}^{a_r} \cdots F_{\beta_l}^{a_l}|a_j
    \in\setN}
  \]

  The other inclusion is shown symmetrically. In the process we also
  proved that the subspace is invariant under left multiplication by
  $F_{\beta_j}$.
\end{proof}
\begin{remark}
  The above lemma shows that $U_v^-(w)$ is an algebra.
\end{remark}


\begin{defn}
  Let $\beta\in \Phi^+$ and let $F_\beta$ be a root vector
  corresponding to $\beta$. Let $u\in U_q$. Define $\ad(F_\beta^i)(u)
  :=[[\dots[u,F_\beta]_v\dots]_v,F_\beta]_v$ and
  $\tilde{\ad}(F_\beta^i)(u) :=[F_\beta,[\dots,[F_\beta,u]_v\dots]]_v$
  where the '$v$-commutator' is taken $i$ times from the left and
  right respectively.
\end{defn}

\begin{prop}
  \label{prop:16}
  Let $u\in (U_A)_\mu$, $\beta\in \Phi^+$ and $F_\beta$ a
  corresponding root vector. Set $r=\left<\mu,\beta^\vee\right>$. Then
  in $U_A$ we have the identity
  \begin{align*}
    \ad(F_\beta^{i})(u) = [i]_\beta! \sum_{n=0}^i (-1)^{n}
    v_\beta^{n(1-i-r)} F_\beta^{(n)} u F_\beta^{(i-n)}
  \end{align*}
  and
  \begin{align*}
    \tilde{\ad}(F_\beta^{i})(u) = [i]_\beta! \sum_{n=0}^i (-1)^{n}
    v_\beta^{n(1-i-r)} F_\beta^{(i-n)} u F_\beta^{(n)}
  \end{align*}
\end{prop}
\begin{proof}
  This is proved by induction. For $i=0$ this is clear. The induction
  step for the first claim:
  \begin{align*}
    &[i]_\beta! \sum_{n=0}^i (-1)^{n} v_\beta^{n(1-i-r)} F_\beta^{(n)}
    u F_\beta^{(i-n)} F_\beta
    \\
    {}&- v_\beta^{-r-2i} F_\beta [i]_\beta! \sum_{n=0}^i (-1)^{n}
    v_\beta^{n(1-i-r)} F_\beta^{(n)} u F_\beta^{(i-n)}
    \\
    =& [i]_\beta! \sum_{n=0}^i (-1)^{n}
    v_\beta^{n(1-i-r)}[i+1-n]F_\beta^{(n)}uF_\beta^{(i+1-n)}
    \\
    &-[i]_\beta! \sum_{n=0}^i (-1)^{n}
    v_\beta^{n(1-i-r)-r-2i}[n+1]F_\beta^{(n+1)}uF_\beta^{(i-n)}
    \\
    =& [i]_\beta! \sum_{n=0}^{i+1} (-1)^{n}
    v_\beta^{n(-i-r)}\left(v_\beta^n[i+1-n]+v_\beta^{n-i-1}[n]\right)F_\beta^{(n)}u
    F_\beta^{(i+1-n)}
    \\
    =& [i+1]_\beta! \sum_{n=0}^{i+1} (-1)^{n} v_\beta^{n(-i-r)}
    F_\beta^{(n)}u F_\beta^{(i+1-n)}.
  \end{align*}
  The other claim is shown similarly by induction.
\end{proof}

So we can define $\ad(F_\beta^{(i)})(u) := ([i]!)\inv
\ad(F_\beta^i)(u)\in U_A$ and $\tilde{\ad}(F_\beta^{(i)})(u) :=
([i]!)\inv \tilde{\ad}(F_\beta^i)(u)\in U_A$.

\begin{prop}
  \label{prop:17}
  Let $a\in\setN$, $u\in (U_A)_\mu$ and
  $r=\left<\mu,\beta^\vee\right>$. In $U_A$ we have the identities
  \begin{align*}
    u F_{\beta}^{(a)} =& \sum_{i=0}^a v_\beta^{(i-a)(r+i)}
    F_\beta^{(a-i)}\ad(F_\beta^{(i)})(u)
    \\
    =& \sum_{i=0}^{a} (-1)^i v_\beta^{a(r+i)-i} F_\beta^{(a-i)}
    \tilde{\ad}(F_\beta^{(i)})(u)
  \end{align*}

  and
  \begin{align*}
    F_{\beta}^{(a)} u =& \sum_{i=0}^a v_\beta^{(i-a)(r+i)}
    \tilde{\ad}(F_\beta^{(i)})(u)F_\beta^{(a-i)}
    \\
    =& \sum_{i=0}^{a} (-1)^i v_\beta^{a(r+i)-i}
    \ad(F_\beta^{(i)})(u)F_\beta^{(a-i)}
  \end{align*}
\end{prop}
\begin{proof}
  This is proved by induction. For $a=0$ this is obvious. The
  induction step for the first claim:
  \begin{align*}
    [a+1]_\beta uF_\beta^{(a+1)} =& u F_\beta^{(a)}F_{\beta}
    \\
    =& \sum_{i=0}^a
    v_\beta^{(i-a)(r+i)}F_\beta^{(a-i)}\ad(F_\beta^{(i)})(u) F_\beta
    \\
    =& \sum_{i=0}^a v_\beta^{(i-a)(r+i)-r-2i}[a+1-i]_\beta
    F_\beta^{(a+1-i)}\ad(F_\beta^{(i)})(u)
    \\
    &+ \sum_{i=0}^a v_\beta^{(i-a)(r+i)}[i+1]_\beta
    F_\beta^{(a-i)}\ad(F_\beta^{(i+1)})(u)
    \\
    =& \sum_{i=0}^a v_\beta^{(i-a-1)(r+i)-i}[a+1-i]_\beta
    F_\beta^{(a+1-i)}\ad(F_\beta^{(i)})(u)
    \\
    &+ \sum_{i=1}^{a+1} v_\beta^{(i-a-1)(r+i-1)}[i]_\beta
    F_\beta^{(a+1-i)}\ad(F_\beta^{(i)})(u)
    \\
    =& \sum_{i=0}^{a+1}
    v_\beta^{(i-a-1)(r+i)}\left(v_\beta^{-i}[a+1-i]_\beta
      +v_\beta^{a+1-i}[i]\right)F_\beta^{(a+1-i)}\ad(F_\beta^{(i)})(u)
    \\
    =&[a+1]_\beta \sum_{i=0}^{a+1} v_\beta^{(i-a-1)(r+i)}
    F_\beta^{(a+1-i)}\ad(F_\beta^{(i)})(u).
  \end{align*}
  So the induction step for the first identity is done.  The three
  other identities are shown similarly by induction.
\end{proof}

Let $s_{i_1}\dots s_{i_N}$ be a reduced expression of $w_0$ and
construct root vectors $F_{\beta_i}$, $i=1,\dots,N$. In the rest of
the section $F_{\beta_i}$ refers to the root vectors constructed as
such. In particular we have an ordering of the root vectors.

\begin{prop}
  \label{prop:18}
  Let $1\leq i< j \leq N$ and $a,b\in \setZ_{>0}$.
  \begin{equation*}
    [F_{\beta_j}^b,F_{\beta_i}^a]_v \in \spa{\setQ(v)}{F_{\beta_i}^{a_i}\cdots F_{\beta_j}^{a_j}|a_l\in \setN, a_i<a, a_j<b}.
  \end{equation*}
\end{prop}
\begin{proof}
  From Theorem~\ref{thm:DP} we get the $a=1$, $b=1$ case. We will
  prove the general case by 2 inductions.

  If $j-i=1$ then $[F_{\beta_j},F_{\beta_i}^a]_v=0$ for all $a$. We
  will use induction over $j-i$.
  
  We have by Proposition~\ref{derivative} that
  \begin{equation*}
    [F_{\beta_j},F_{\beta_i}^a]_v = v^{-((a-1)\beta_i|\beta_j)}F_{\beta_i}^{a-1} [F_{\beta_j},F_{\beta_i}]_v  + [F_{\beta_j},F_{\beta_i}]_v F_{\beta_i}^{a-1}.
  \end{equation*}
  The first term is in the correct subspace by
  Theorem~\ref{thm:DP}. On the second we use the fact that
  $[F_{\beta_i},F_{\beta_j}]_v$ only contains factors
  $F_{\beta_{i+1}}^{a_i}\cdots F_{\beta_{j-1}}^{a_{j-1}}$ and the
  induction over $j-i$ as well as induction over $a$ to conclude that
  we can commute the $F_{\beta_i}^{a-1}$ to the correct place and be
  in the correct subspace.
  
  Now just make a similar kind of induction on $i-j$ and $b$ to get
  the result that
  \begin{equation*}
    [F_{\beta_j}^b,F_{\beta_i}^a]_v \in \spa{\setQ(v)}{F_{\beta_i}^{a_i}\cdots F_{\beta_j}^{a_j}|a_l\in \setN, a_i<a, a_j<b}.
  \end{equation*}
\end{proof}

\begin{cor}
  \label{cor:2}
  Let $1\leq i< j \leq N$ and $a,b\in \setZ_{>0}$.
  \begin{equation*}
    [F_{\beta_j}^{(b)},F_{\beta_i}^{(a)}]_v \in \spa{A}{F_{\beta_i}^{(a_i)}\cdots F_{\beta_j}^{(a_j)}|a_l\in \setN, a_i<a, a_j<b}.
  \end{equation*}
\end{cor}
\begin{proof}
  Proposition~\ref{prop:18} tells us that there exists $c_k\in
  \setQ(v)$ such that
  \begin{equation*}
    [F_{\beta_j}^{(b)},F_{\beta_i}^{(a)}]_v = \sum_k c_kF_{\beta_i}^{(a_i^k)}\cdots F_{\beta_j}^{(a_j^k)}
  \end{equation*}
  with $a_i^k<a$ and $a_j^k<b$ for all $k$. But since
  $[F_{\beta_j}^{(b)},F_{\beta_i}^{(a)}]_v \in U_A^-$ there exists
  $b_k \in A$ such that
  \begin{equation*}
    [F_{\beta_j}^{(b)},F_{\beta_i}^{(a)}]_v = \sum_k b_kF_{\beta_1}^{(a_1^k)}\cdots F_{\beta_N}^{(a_N^k)}.
  \end{equation*}
  Now we have two expressions of
  $[F_{\beta_j}^{(b)},F_{\beta_i}^{(a)}]_v$ in terms of a basis of
  $U_{\setQ(v)}^-$. So we must have that the $c_k$'s are equal to the
  $b_k$'s. Hence $c_k\in A$ for all $k$
\end{proof}

\begin{lemma}
  \label{lemma:22}
  Let $n\in \setN$. Let $1\leq j<k\leq N$.
  
  $\ad(F_{\beta_j}^{(i)})(F_{\beta_k}^{(n)})=0$ and
  $\tilde{\ad}(F_{\beta_k}^{(i)})(F_{\beta_j}^{(n)})=0$ for $i\gg 0$.
\end{lemma}
\begin{proof}
  We will prove the first assertion. The second is proved completely
  similar. We can assume $\beta_j=1$ because
  \begin{equation*}
    \ad(F_{\beta_j}^{(i)})(F_{\beta_k}^{(n)}) = T_{s_{i_1}}\cdots
    T_{s_{i_{j-1}}}\left( \ad(F_{\alpha_{i_j}}^{(i)})( T_{s_{i_j}}\cdots
      T_{s_{i_{k-1}}}( F_{\alpha_{i_k}}^{(n)}))\right).
  \end{equation*}
  So we assume $\beta_j=\beta_1=:\beta\in \Pi$ and $\alpha:=\beta_k =
  s_{i_1}\dots s_{i_{k-1}}(\alpha_{i_j})\in \Phi^+$. We have
  \begin{equation*}
    \ad(F_\beta)(F_\alpha^{(n)})\in
    \spa{A}{F_{\beta_2}^{(a_2)}\cdots F_{\beta_k}^{(a_k)}|a_l\in \setN,
      a_k<n},
  \end{equation*}
  hence the same must be true for
  $\ad(F_\beta^{(i)})(F_\alpha^{(n)})$. By homogenity if the monomial
  $F_{\beta_2}^{(a_2)}\cdots F_{\beta_k}^{(a_k)}$ appears with nonzero
  coefficient then we must have
  \begin{equation*}
    i \beta + n\alpha = \sum_{s=2}^k a_s \beta_s
  \end{equation*}
  or equivalently
  \begin{equation*}
    (n-a_k)\alpha = \sum_{s=2}^{k-1} a_s \beta_s - i\beta.
  \end{equation*}
  Use $s_\beta$ on this to get
  \begin{equation*}
    (n-a_k)s_\beta(\alpha) = \sum_{s=2}^{s-1} a_s s_\beta(\beta_s) + i\beta.
  \end{equation*}
  By the way the $\beta_s$'s are chosen $s_\beta(\beta_s)>0$ for
  $1<s<k$. So this implies that a positive multiple ($n-a_j$) of a
  positive root must have $i\beta$ as coefficient. If we choose $i$
  greater than $nd$ where $d$ is the maximal possible coefficient of a
  simple root in any positive root then this is not possible. Hence we
  must have for $i>nd$ that $\ad(F_\beta^{(i)})(F_\alpha^{(n)})=0$.
\end{proof}

In the next lemma we will need to work with inverse powers of some of
the $F_\beta$'s. We know from e.g.~\cite{HHA-kvante} that
$\{F_{\alpha}^{a}|a\in\setN\}$, $\alpha\in \Pi$ is a multiplicative
set so we can take the Ore localization in this set. Since $R_{w}$ is
an algebra isomorphism of $U_v$ we can also take the Ore localization
in one of the 'root vectors' $F_{\beta_j}$. We will denote the Ore
localization in $F_{\beta}$ by $U_{v(F_\beta)}$.

\begin{lemma}
  \label{sec:twisting-functors-6}
  Let $\beta\in \Phi^+$ and $F_\beta$ a root vector. Let $u\in
  (U_v)_\mu$ be such that $\tilde{\ad}(F_\beta^i)(u)=0$ for $i\gg
  0$. Let $a\in\setN$ and set $r=\left< \mu,\beta^\vee\right>$. Then
  in the algebra $U_{v(F_{\beta})}$ we get
  \begin{equation*}
    uF_{\beta}^{-a}= \sum_{i\geq 0} v_\beta^{-ar - (a+1)i}{a+i-1\brack i}_\beta  F_{\beta}^{-i-a}\tilde{\ad}(F_\beta^i)(u)
  \end{equation*}
  and if $u'\in (U_v)_\mu$ is such that $\ad(F_\beta^i)(u')=0$ for
  $i\gg 0$
  \begin{equation*}
    F_{\beta}^{-a} u'= \sum_{i\geq 0} v_\beta^{-ar - (a+1)i}{a+i-1\brack i}_\beta  \ad(F_\beta^i)(u')F_{\beta}^{-i-a}.
  \end{equation*}
\end{lemma}
\begin{proof}
  First we want to show that
  \begin{equation}
    \tilde{\ad}(F_{\beta}^i)(u)F_{\beta}\inv=\sum_{k=i}^\infty v_\beta^{-r-2k}F_{\beta}^{-k+i-1}\tilde{\ad}(F_{\beta}^k)(u).
  \end{equation}
  Remember that $\tilde{\ad}(F_{\beta}^k)(u)=0$ for $k$ big enough so
  this is a finite sum. This is shown by downwards induction on
  $i$. If $i$ is big enough this is $0=0$. We have
  \begin{equation*}
    F_\beta \tilde{\ad}(F_\beta^i)(u) = \tilde{\ad}(F_{\beta}^{i+1})(u) + v_\beta^{-r-2i} \tilde{\ad}(F_\beta^i)(u)F_\beta
  \end{equation*}
  so
  \begin{align*}
    \tilde{\ad}(F_\beta^i)(u) F_\beta\inv =& F_\beta\inv
    \tilde{\ad}(F_{\beta}^{i+1})(u) F_\beta\inv+ v_\beta^{-r-2i}
    F_\beta\inv \tilde{\ad}(F_\beta^i)(u)
    \\
    =& \sum_{k=i+1}^\infty
    v_\beta^{-r-2k}F_{\beta}^{-k+i-1}\tilde{\ad}(F_{\beta}^k)(u) +
    v_\beta^{-r-2i} F_\beta\inv \tilde{\ad}(F_\beta^i)(u)
    \\
    =& \sum_{k=i}^\infty
    v_\beta^{-r-2k}F_{\beta}^{-k+i-1}\tilde{\ad}(F_{\beta}^k)(u).
  \end{align*}

  Setting $i=0$ in the above we get the induction start:
  \begin{equation*}
    u F_{\beta}\inv = \sum_{k\geq 0} v_\beta^{-r-2k} F_\beta^{-k-1} \tilde{\ad}(F_\beta^k)(u).
  \end{equation*}
  
  For the induction step assume
  \begin{equation*}
    u F_\beta^{-a} = \sum_{i\geq 0} v_\beta^{-ar - (a+1)i}{a+i-1 \brack i}_\beta F_\beta^{-a-i} \tilde{\ad}(F_\beta^i)(u).
  \end{equation*}

  Then
  \begin{align*}
    u F_\beta^{-a-1} =& \sum_{i\geq 0} v_\beta^{-ar - (a+1)i}{a+i-1
      \brack i}_\beta F_\beta^{-a-i} \tilde{\ad}(F_\beta^i)(u)
    F_\beta\inv
    \\
    =& \sum_{i\geq 0} v_\beta^{-ar - (a+1)i}{a+i-1 \brack i}_\beta
    F_\beta^{-a-i} \sum_{k\geq i}
    v_\beta^{-r-2k}F_{\beta}^{-k+i-1}\tilde{\ad}(F_{\beta}^k)(u)
    \\
    =& \sum_{k\geq 0} \sum_{i=0}^k v_\beta^{-(a+1)r -(a+1)i-2k}{a+i-1
      \brack i}_\beta F_\beta^{-a-1-k}\tilde{\ad}(F_\beta^k)(u)
    \\
    =& \sum_{k\geq 0} v_\beta^{-(a+1)r -(a+2)k}\left( \sum_{i=0}^k
      v_\beta^{-(a+1)i+a k}{a+i-1 \brack i}_\beta\right)
    F_\beta^{-a-1-k}\tilde{\ad}(F_\beta^k)(u).
  \end{align*}

  The induction is finished by observing that
  \begin{align*}
    \sum_{i=0}^k v_\beta^{-(a+1)i+a k}{a+i-1 \brack i}_\beta =&
    v_\beta^{ak}+ \sum_{i=1}^k v_\beta^{-(a+1)i+ak} \left( v_\beta^i
      {a+i\brack i}_\beta - v_\beta^{a+i}{a+i-1\brack
        i-1}_\beta\right)
    \\
    =&v_\beta^{ak}+ \sum_{i=1}^k v_\beta^{-ai +ak}{a+i\brack i}_\beta
    - \sum_{i=1}^k v_\beta^{-a(i-1)+ak} {a+i-1\brack i-1}_\beta
    \\
    =& v_\beta^{ak}+ \sum_{i=1}^k v_\beta^{-ai +ak}{a+i\brack i}_\beta
    - \sum_{i=0}^{k-1} v_\beta^{-ai +ak} {a+i\brack i}_\beta
    \\
    =& {a+k\brack k}_\beta.
  \end{align*}

  The other identity is shown similarly by induction.
\end{proof}

\begin{defn}
  Let $\beta\in \Phi^+$ and let $\beta$ be $F_{\beta}$ a root
  vector. We define for $n\in \setN$ in $U_{v(F_\beta)}$
  \begin{equation*}
    F_\beta^{(-n)}= [n]! F_\beta^{-n}
  \end{equation*}
  i.e. $F_\beta^{(-n)}=\left(F_\beta^{(n)}\right)\inv$.
\end{defn}

\begin{cor}
  \label{cor:3}
  Let $\beta\in \Phi^+$ and $F_\beta$ a root vector. Let $u\in
  (U_v)_\mu$ be such that $\tilde{\ad}(F_\beta^{(i)})(u)=0$ for $i\gg
  0$. Let $a\in\setN$ and set $r=\left< \mu,\beta^\vee\right>$. Then
  in the algebra $U_{v(F_{\beta})}$ we get
  \begin{equation*}
    uF_{\beta}^{(-a)}F_\beta\inv = \sum_{i\geq 0} v_\beta^{-(a+1)r - (a+2)i}  F_{\beta}^{(-i-a)}F_\beta\inv \tilde{\ad}(F_\beta^{(i)})(u)
  \end{equation*}
  and if $u'\in (U_v)_\mu$ is such that $\ad(F_\beta^{(i)})(u')=0$ for
  $i\gg 0$
  \begin{equation*}
    F_{\beta}^{(-a)}F_{\beta}\inv u'= \sum_{i\geq 0} v_\beta^{-(a+1)r - (a+2)i}  \ad(F_\beta^{(i)})(u')F_{\beta}^{(-i-a)}F_\beta\inv.
  \end{equation*}
\end{cor}

\section{Twisting functors}
\label{sec:twisting-functors}

In this paper we are following the paper~\cite{HHA-kvante}
closely. The definition of twisting functors for quantum group modules
given later and the ideas in this section are mostly coming from this
paper.

We will start by showing that the semiregular bimodule $S_v^w$ is a
bimodule isomorphic to $U_v^-(w)^* \otimes_{U_v^-(w)}U_v$ as a right
module.

Recall how $U_v(w)$, $S_v^w$ and $S_v(F)$ are defined: Let
$s_{i_r}\cdots s_{i_1}$ be a reduced expression for $w$ and
$F_{\beta_j}=R_{s_{i_1}}\cdots R_{s_{i_{j-1}}}(F_{\alpha_{i_j}})$ as
usual then
\begin{equation*}
  U_v^-(w)=\spa{\setQ(v)}{F_{\beta_1}^{a_1}\cdots F_{\beta_r}^{a_r}|a_i\in \setN},
\end{equation*}
\begin{equation*}
  S_v^w=U_v\tensor_{U_v^-(w)}U_v^-(w)^*
\end{equation*}
and for $F\in U_v^-$ such that $\{F^{a}|a\in\setN\}$ is a
multiplicative set
\begin{equation*}
  S_v(F)=U_{v(F)}/U_v
\end{equation*}
where $U_{v(F)}$ denotes the Ore localization in the multiplicative
set $\{F^{a}|a\in\setN\}$.

In the following proposition we will define a left $U_v$ isomorphism
between $S_v^w$ and $S_v(F_{\beta_r})\tensor_{U_v}S_v^{w'}$ where
$w'=s_{i_r}w$. We will need some notation. Let $m\in\setN$. We denote
by $f_{m}^{(r)}\in (\setQ(v)[F_{\beta_r}])^*$ the linear function
defined by $f_{m}^{(r)}(F_{\beta_r}^a)=\delta_{m,a}$. We will drop the
$(r)$ from the notation in most of the following. For $g\in
U_v^-(w')^*$ we define $f_{m}\cdot g$ to be the linear function
defined by: For $x\in U_v^-(w')$, $(f_{m}\cdot
g)(xF_{\beta_r}^a)=f_{m}(F_{\beta_r}^a)g(x)$. From the definition of
$U_v^-(w)$ and because we are taking \emph{graded} dual every $f\in
U_v^-(w)^*$ is a linear combination of functions on the form
$f_{m}\cdot g$ for some $m\in\setN$ and $g\in U_v^-(w')$ (by induction
this implies that every function in $U_v^-(w)$ is a linear combination
of functions of the form $f_{m_r}^{(r)}\cdots f_{m_2}^{(2)} \cdot
f_{m_1}^{(1)}$ for some $m_1,\dots,m_r\in \setN$). Note that the
definition of $f_m$ makes sense for $m<0$ but then $f_m=0$.

\begin{prop}
  \label{sec:twisting-functors-7}
  Assume $w=s_{i_k}\cdots s_{i_1}=s_{i_k}w'$, where $k$ is the length
  of $w$, then as a left $U_v$ module
  \begin{equation*}
    S_v^w\iso S_v(F_{\beta_k})\otimes_{U_v} S_v^{w'}
  \end{equation*}
  by the following left $U_v$ isomorphism
  \begin{equation*}
    \phi_k: S_v^w\to S_v(F_{\beta_k})\otimes_{U_v} S_v^{w'}
  \end{equation*}
  defined by:
  \begin{equation*}
    \phi_k(u\tensor f_{m}\cdot g)=uF_{\beta_k}^{-m-1}K_{\beta_k}\tensor (1\tensor g), \quad u\in U_v, m\in\setN, g\in U_v^-(w')^*.
  \end{equation*}
  The inverse to $\phi_k$ is the left $U_v$-homomorphism
  $\psi_k:S_v(F_{\beta_k})\otimes_{U_v} S_v^{w'} \to S_v^w$ given by:
  \begin{equation*}
    \psi_k(uF_{\beta_k}^{-m}\tensor (1\tensor g)) = v^{(m\beta_k|\beta_k)}uK_{\beta_k}\inv \tensor f_{m-1}\cdot g, \quad u\in U_v. m\in\setN, g\in U_v^-(w')^*.
  \end{equation*}
\end{prop}
\begin{proof}
  The question is if $\phi_k$ is welldefined. Let $f=f_m\cdot g$. We
  need to show that the recipe for $uF_{\beta_j}\tensor f$ is the same
  as the recipe for $u\tensor F_{\beta_j}f$ for $j=1,\dots,k$. For
  $j=k$ this is easy to see. Assume from now on that $j<k$. We need to
  figure out what $F_{\beta_j}f$ is. We have by
  Proposition~\ref{prop:17} (setting
  $r=\left<\beta_j,\beta_k^\vee\right>$)
  \begin{align*}
    (F_{\beta_j}f)(xF_{\beta_k}^a)&=f(xF_{\beta_k}^aF_{\beta_j})
    \\
    &=f\left( x \sum_{i=0}^a v_\beta^{(i-a)(r+i)} {a\brack i}_\beta
      \tilde{\ad}(F_{\beta_k}^{i})(F_{\beta_j})F_{\beta_k}^{a-i}
    \right)
    \\
    &=\left( \sum_{i=0}^a v_\beta^{-m(r+i)} {m+i \brack i}_\beta
      f_{m+i}\cdot
      \left(\tilde{\ad}(F_{\beta_k}^{i})(F_{\beta_j})g\right)
    \right)(xF_{\beta_k}^a)
    \\
    &= \left( \sum_{i\geq 0} v_\beta^{-m(r+i)} {m+i\brack i}_\beta
      f_{m+i}\cdot
      \left(\tilde{\ad}(F_{\beta_k}^{i})(F_{\beta_j})g\right)
    \right)(xF_{\beta_k}^a)
  \end{align*}
  so
  \begin{equation*}
    F_{\beta_j}f= \sum_{i\geq 0} v_\beta^{-m(r+i)} {m+i\brack i}_\beta f_{m+i}\cdot \left(\tilde{\ad}(F_{\beta_k}^{i})(F_{\beta_j})g\right).
  \end{equation*}
  Note that the sum is finite because of Lemma~\ref{lemma:22}.
  
  On the other hand we have that $uF_{\beta_j}\tensor f$ is sent to
  (using Lemma~\ref{sec:twisting-functors-6})
  \begin{align*}
    &uF_{\beta_j}F_{\beta_k}^{-m-1}K_{\beta_k}\tensor (1\tensor g)
    \\
    &= u \sum_{i\geq 0} v_{\beta_k}^{-(m+1)r - (m+2)i}{m+i \brack
      i}_\beta F_{\beta_k}^{-i-m-1}
    \tilde{\ad}(F_{\beta_k}^i)(F_{\beta_j}) K_{\beta_k} \tensor (1
    \tensor g)
    \\
    &= u \sum_{i\geq 0} v_{\beta_k}^{-mr - mi}{m+i \brack i}_\beta
    F_{\beta_k}^{-i-m-1}K_{\beta_k}
    \tilde{\ad}(F_{\beta_k}^i)(F_{\beta_j}) \tensor (1 \tensor g).
  \end{align*}
  Using the fact that $\tilde{\ad}(F_{\beta_k}^i)(F_{\beta_j})$ can be
  moved over the first and the second tensor we see that the two
  expressions $uF_{\beta_j}\tensor f$ and $u\tensor F_{\beta_j}f$ are
  sent to the same.

  So $\phi_k$ is a welldefined homomorphism. It is clear from the
  construction that $\phi_k$ is a $U_v$ homomorphism.
  
  We also need to prove that $\psi_k$ is welldefined. We prove that
  $uF_{\beta_k}^{-m}F_{\beta_j}\tensor (1 \tensor g)$ is sent to the
  same as $uF_{\beta_k}^{-m}\tensor (1 \tensor F_{\beta_j}g)$ by
  induction over $k-j$. If $j=k-1$ we see from
  Lemma~\ref{sec:twisting-functors-6} and Theorem~\ref{thm:DP} that
  $F_{\beta_{k-1}}F_{\beta_k}^{-a}=v^{-(a\beta_k|\beta_{k-1})}F_{\beta_k}^{-a}F_{\beta_{k-1}}$
  and therefore $uF_{\beta_k}^{-m}F_{\beta_{k-1}}\tensor (1 \tensor
  g)$ is sent to
  \begin{align*}
    &v^{(m\beta_k-\beta_j|\beta_k)+(m\beta_k|\beta_{k-1})} u
    K_{\beta_k}\inv F_{\beta_{k-1}} \tensor f_{m-1}\cdot g
    \\
    &=v^{(m\beta_k+(m-1)\beta_{k-1}\beta_k)} u K_{\beta_k}\inv \tensor
    F_{\beta_{k-1}}(f_{m-1}\cdot g).
  \end{align*}
  Note that because we have
  $\tilde{\ad}(F_{\beta_k}^i)(F_{\beta_j})=0$ for all $i\geq 1$ we get
  $F_{\beta_{k-1}}(f_{m-1}\cdot g)=
  v^{-(\beta_{k-1}|(m-1)\beta_k)}f_{m-1}\cdot
  (F_{\beta_{k-1}}g)$. Using this we see that
  $uF_{\beta_k}^{-m}F_{\beta_{k-1}}\tensor (1 \tensor g)$ is sent to
  the same as $uF_{\beta_k}^{-m}\tensor (1 \tensor F_{\beta_{k-1}}g)$.

  Now assume $j-k>1$.  To calculate what
  $uF_{\beta_k}^{-m}F_{\beta_j}\tensor (1 \tensor g)$ is sent to we
  need to calculate $F_{\beta_k}^{-m}F_{\beta_j}$. By
  Lemma~\ref{sec:twisting-functors-6}
  \begin{align*}
    F_{\beta_k}^{-m}F_{\beta_j} = v^{mr}F_{\beta_j}F_{\beta_k}^{-m} -
    \sum_{i\geq 1} v_\beta^{-(m+1)i}{m+i-1\brack i}_\beta
    F_{\beta_k}^{-m-i}\tilde{\ad}(F_{\beta_k}^i)(u).
  \end{align*}
  So
  \begin{align*}
    uF_{\beta_k}^{-m}F_{\beta_j} \tensor (1 \tensor g) = u\Bigg(&
      v_\beta^{mr}F_{\beta_j}F_{\beta_k}^{-m} 
      \\
      &- \sum_{i\geq 1}
      v_\beta^{-(m+1)i}{m+i-1\brack i}_\beta
      F_{\beta_k}^{-m-i}\tilde{\ad}(F_{\beta_k}^i)(u)\Bigg) \tensor
    (1\tensor g).
  \end{align*}
  
  By the induction over $k-j$ (remember that
  $\tilde{\ad}(F_{\beta_k}^i)(u)$ is a linear combination of ordered
  monomials involving only the elements $F_{\beta_{j+1}}\cdots
  F_{\beta_{k-1}}$) this is sent to the same as
  \begin{equation*}
    u\left(
      v_\beta^{mr} F_{\beta_j}F_{\beta_k}^{-m}\tensor (1\tensor
      g)-\sum_{i\geq 1} v_\beta^{-(m+1)i}{m+i-1\brack i}_\beta F_{\beta_k}^{-m-i}\tensor (1\tensor \tilde{\ad}(F_{\beta_k}^i)(u) g) \right)
  \end{equation*}
  which is sent to
  \begin{equation*}
    \begin{split}
      u \Bigg(& v_\beta^{mr+2m} F_{\beta_j}K_{\beta_k}\inv \tensor
      f_{m-1}\cdot g
      \\
      &-K_{\beta_k}\inv \tensor \sum_{i\geq 1}
      v_\beta^{2(m+i)-(m+1)i}{m+i-1\brack i}_\beta \tensor
      f_{m+i-1}\cdot (\tilde{\ad}(F_{\beta_k}^i)(u) g) \Bigg)
      \\
      = v_\beta^{2m} uK_{\beta_k}\inv \Bigg(& v_\beta^{(m-1)r}
      F_{\beta_j}\tensor f_{m-1}\cdot g
      \\
      &- 1 \tensor \sum_{i\geq 1} v_\beta^{-(m-1)i}{m+i-1\brack
        i}_\beta \tensor f_{m+i-1}\cdot (\tilde{\ad}(F_{\beta_k}^i)(u)
      g) \Bigg)
      \\
      =v^{(m\beta_k|\beta_k)}u &K_{\beta_k}\inv \tensor f_{m-1}\cdot
      (F_{\beta_j}g).
    \end{split}
  \end{equation*}
  But this is what $uF_{\beta_k}^{-m}\tensor (1\tensor F_{\beta_j}g)$
  is sent to. We have shown by induction that $\psi_k$ is well
  defined. It is easy to check that $\psi_k$ is the inverse to
  $\phi_k$.
\end{proof}

\begin{prop}
  \label{sec:twisting-functors-4}
  Let $s_{i_r}\cdots s_{i_1}$ be a reduced expression of $w\in
  W$. There exists an isomorphism of left $U_v$-modules
  \begin{equation*}
    S_v^w\iso S_v(F_{\beta_r})\otimes_{U_v} \cdots \tensor_{U_v} S_v(F_{\beta_1})
  \end{equation*}
\end{prop}
\begin{proof}
  The proof is by induction of the length of $w$. Note that
  $S_v^e=U_v\tensor_k k^*\iso U_v$ so
  Proposition~\ref{sec:twisting-functors-7} with $w'=e$ gives the
  induction start.

  Assume the length of $w$ is $r>1$.  By
  Proposition~\ref{sec:twisting-functors-7} we have $S_v^w\iso
  S_v(F_{\beta_r})\tensor_{U_v}S_v^{w'}$. By induction $S_v^{w'}\iso
  S_v(F_{\beta_{r-1}})\tensor_{U_v} \cdots \tensor_{U_v}
  S_v(F_{\beta_1})$. This finishes the proof.
\end{proof}

We can now define a right action on $S_v^w$ by the isomorphism in
Proposition~\ref{sec:twisting-functors-4}. By first glance this might
depend on the chosen reduced expression for $w$. But the next
proposition proves that this right action does not depend on the
reduced expression chosen.

\begin{prop}
  \label{prop:7}
  As a right $U_v$ module $S_v^w \iso U_v^-(w)^*\tensor_{U_v} U_v$.
\end{prop}
\begin{proof}
  All isomorphisms written in this proof are considered to be right
  $U_v$ isomorphisms. This is proved in a very similar way to
  Proposition~\ref{sec:twisting-functors-7}. We will sketch the proof
  here.

  For $l\in\{1,\dots,N\}$ define $S_v^l=(U_v^l)^*\tensor_{U_v^l} U_v$
  where $U_v^l=\spa{\setQ(v)}{F_{\beta_l}^{a_l}\cdots
    F_{\beta_r}^{a_r}|a_i\in\setN}$. Note that
  $S_v^1=U_v^-(w)^*\tensor_{U_v} U_v$. We want to show that
  $(U_v^l)^*\tensor_{U_v^l} U_v\iso S_v^{l+1}\tensor_{U_v}
  S_v(F_{\beta_l})$. If we prove this we will have $S_v^1\iso
  S_v^2\tensor_{U_v} S_v(F_{\beta_1}) \iso \cdots \cdots \iso
  S_v(F_{\beta_r})\tensor_{U_v} \cdots \tensor_{U_v}
  S_v(F_{\beta_1})\iso S_v^w$ as a right module and we are done.

  Let $r=\left<\beta_j,\beta_l^\vee\right>$. From
  Proposition~\ref{prop:17} we have
  \begin{equation*}
    F_{\beta_j} F_{\beta_l}^a = \sum_{i=0}^a v_\beta^{(i-a)(r+i)} {a\brack i}_\beta F_{\beta_l}^{a-i}\ad(F_{\beta_l}^{i})(F_{\beta_j})
  \end{equation*}
  
  and by Lemma~\ref{sec:twisting-functors-6} we have
  \begin{equation*}
    F_{\beta_l}^{-a} F_{\beta_j}= \sum_{i\geq 0} v_{\beta_l}^{-ar - (a+1)i}{a+i-1\brack i}_{\beta_l}  \ad(F_{\beta_l}^i)(F_{\beta_r})F_{\beta_l}^{-i-a}.
  \end{equation*}

  We define the right homomorphism $\phi_l$ from
  $(U_v^l)^*\tensor_{U_v^l} U_v$ to $ S_v^{l+1}\tensor_{U_v}
  S_v(F_{\beta_l})$ by
  \begin{equation*}
    \phi_l(g\cdot f_{m_l}\tensor u) = (g\tensor 1)\tensor K_{\beta_l}F_{\beta_l}^{-m_l-1}u.
  \end{equation*}
  Like in the previous propisition we can use the above formulas to
  show that this is well defined and we can define an inverse like in
  the previous proposition only reversed.  The inverse is:
  \begin{equation*}
    \psi_l((g\tensor 1) \tensor F_{\beta_l}^{-m-1}u) = v^{-((m+1)\beta_l|\beta_l)}g\cdot f_{m}\tensor K_{\beta_l}\inv u.
  \end{equation*}
\end{proof}

So we have now that $S_v^w$ is a bimodule isomorphic to
$U_v\tensor_{U_v^-(w)}U_v^-(w)^*$ as a left module and isomorphic to
$U_v^-(w)^*\tensor_{U_v^-(w)}U_v$ as a right module. We want to
examine the isomorphism between these two modules. For example what is
the left action of $K_{\alpha}$ on $f\tensor 1\in
(U_v^-(w))^*\tensor_{U_v^-(w)}U_v$.

Assume $f=f_{m_r}^{(r)}\cdots f_{m_1}^{(1)}$ i.e. that
$f(F_{\beta_1}^{a_1}\cdots F_{\beta_r}^{a_r})=\delta_{m_1,a_1}\cdots
\delta_{m_r,a_r}$. Then we get via the isomorphism
$(U_v^-(w))^*\tensor_{U_v^-(w)}U_v\iso
S_v(F_{\beta_r})\tensor_{U_v}\cdots \tensor_{U_v} S_v(F_{\beta_1})$
that $f\tensor u$ is sent to
\begin{equation*}
  K_{\beta_r}F_{\beta_r}^{-m_r-1} \tensor \cdots \tensor K_{\beta_1}F_{\beta_1}^{-m_1-1} u.
\end{equation*}
We want to investigate what this is sent to under the isomorphism
$S_v(F_{\beta_r})\tensor_{U_v}\cdots \tensor_{U_v}
S_v(F_{\beta_1})\iso U_v\tensor_{U_v^-(w)}(U_v^-(w))^*$. To do this we
need to commute $u$ with $F_{\beta_1}^{-m_1-1}$, then
$F_{\beta_2}^{-m_2-1}$ and so on. So we need to find $\tilde{u}$ and
$m_1',\dots,m_r'$ such that
\begin{equation*}
  K_{\beta_r}F_{\beta_r}^{-m_r-1}\cdots K_{\beta_1}F_{\beta_1}^{-m_1-1}u=\tilde{u}K_{\beta_r}F_{\beta_r}^{-m_r'-1}\cdots K_{\beta_1}F_{\beta_1}^{-m_1'-1}
\end{equation*}
or equivalently
\begin{equation*}
  uF_{\beta_1}^{m_1'+1}K_{\beta_1}\inv \cdots F_{\beta_r}^{m_r'+1}K_{\beta_r}\inv = F_{\beta_1}^{m_1+1}K_{\beta_1}\inv \cdots F_{\beta_r}^{m_r+1}K_{\beta_r}\inv \tilde{u}.
\end{equation*}
Assume we have found such $\tilde{u}$ and $m_1',\dots,m_r'$ then the
above tensor is sent to
\begin{equation*}
  v^{\sum_{i=1}^r ((m'+1)\beta_i|\beta_i)} \tilde{u} \tensor \tilde{f}
\end{equation*}
where $\tilde{f}=f_{m_r'}^{(r)}\cdots f_{m_1'}^{(1)}$. So in
conclusion we have that $f\tensor u\in
(U_v^-(w))^*\tensor_{U_v^-(w)}U_v$ maps to $v^{\sum_{i=1}^r
  ((m'+1)\beta_i|\beta_i)} \tilde{u} \tensor \tilde{f}\in
U_v\tensor_{U_v^-(w)}(U_v^-(w))^*$ where $\tilde{f}$ and $\tilde{u}$
are defined as above.

We have a similar result the other way: $u\tensor f\in
U_v\tensor_{U_v^-(w)}(U_v^-(w))^*$ maps to $v^{-\sum_{i=1}^r
  ((m+1)\beta_i|\beta_i)} \bar{u} \tensor \bar{f}\in (U_v^-(w))^*
\tensor_{U_v^-(w)} U_v$. So if we want to figure out the left action
of $u$ on a tensor $f\tensor 1$ we need to first use the isomorpism
$(U_v^-(w))^*\tensor_{U_v^-(w)}U_v\to
U_v\tensor_{U_v^-(w)}(U_v^-(w))^*$ then use $u$ on this and then use
the isomorphism $U_v\tensor_{U_v^-(w)}(U_v^-(w))^*\to
(U_v^-(w))^*\tensor_{U_v^-(w)}U_v$ back again.

In particular if $u=K_{\alpha}$ we have $\bar{f}=f$ and
$\bar{u}=v^{\sum_{i=1}^r ((m_i+1)\beta_i|\beta_i)}K_\alpha$. Note that
if $f=f_{m_r}^{(m_r)}\cdots f_{m_1}^{(1)}$ then the grading of $f$ is
$\sum_{i=1}^r m_i\beta_i$ so $K_{\alpha}(f\tensor
1)=v^{\left(\gamma+\sum_{i=1}^r \beta_i|\alpha\right)}f\tensor
K_{\alpha}$ for $f\in (U_v^-(w))^*_\gamma$.

\begin{defn}
  Let $w\in W$. For a $U_v$-module $M$ define a 'twisted' version of
  $M$ called $^w M$. The underlying space is $M$ but the action on $^w
  M$ is given by: For $m\in M$ and $u\in U_v$
  \begin{equation*}
    u\cdot m = R_{w\inv}(u)m.
  \end{equation*}
\end{defn}
Note that if $w,s\in W$ and $l(sw)>l(w)$ then $^{s}(^w M)= {^{sw}} M$
since for $u\in U_v$ and $m\in {^{s}}(^w M)$: $u \cdot m=R_{s}(u)\cdot
m=R_{w\inv}(R_{s}(u))m=R_{(sw)\inv}(u)m$.

\begin{defn}
  The twisting functor $T_w$ associated to an element $w\in W$ is the
  following:
  
  $T_w: U_v-\Mod \to U_v-\Mod$ is an endofunctor on $U_v-\Mod$. For a
  $U_v$-module $M$:
  \begin{equation*}
    T_wM={^w}(S_v^w\tensor_{U_v} M).
  \end{equation*}
\end{defn}

\begin{defn}
  \label{defn:asd1}
  Let $M$ be a $U_v$-module and $\lambda: U_v^0 \to \setQ(v)$ a
  character (i.e. an algebra homomorphism into $\setQ(v)$). Then
  \begin{equation*}
    M_\lambda = \{ m\in M | \forall u\in U_v^0, u m = \lambda(u)m\}.
  \end{equation*}
  Let $X$ denote the set of characters.  Let $\wt M$ denote all the
  weights of $M$, i.e. $\wt M = \{ \lambda\in X | M_\lambda \neq 0
  \}$.  We define for $\mu\in \Lambda$ the character $v^\mu$ by
  $v^\mu(K_\alpha) = v^{(\mu|\alpha)}$. We also define $v_\beta^\mu =
  v^{\frac{(\beta|\beta)}{2} \mu}$.  We say that $M$ only has integral
  weights if all its weights are of the form $v^\mu$ for some $\mu \in
  \Lambda$.
\end{defn}

$W$ acts on $X$ by the following: For $\lambda\in X$ define $w\lambda$
by
\begin{equation*}
  (w\lambda)(u) = \lambda(R_{w\inv}(u)).
\end{equation*}
Note that $w v^\mu = v^{w(\mu)}$.

We will also need the dot action. It is defined as such: For a weight
$\mu\in X$ and $w\in W$, $w.\mu=v^{-\rho}w(v^\rho \mu)$ where
$\rho=\frac{1}{2}\sum_{\beta\in\Phi} \beta$ as usual.  The Verma
module $M(\lambda)$ for $\lambda\in X$ is defined as
$M(\lambda)=U_v\tensor_{U_v^{\geq 0}} \setQ(v)_\lambda$ where
$\setQ(v)_\lambda$ is the onedimensional module with trivial $U_v^+$
action and $U_v^0$ action by $\lambda$ (i.e. $K_\mu\cdot 1 =
\lambda(K_\mu)$). $M(\lambda)$ is a highest weight module generated by
$v_\lambda = 1\tensor 1$.

Note that $R_{w\inv}$ sends a weight space of weight $\mu$ to the
weight space of weight $w(\mu)$ since if we have a vector $m$ with
weight $\mu$ in a module $M$ we get in ${^w}M$ that
\begin{equation*}
  K_\alpha\cdot m=R_{w\inv}(K_\alpha)m=K_{w\inv(\alpha)}m=v^{(w\inv(\alpha)|\mu)}m=v^{(\alpha|w(\mu))}m.
\end{equation*}

We define the character of a $U_v$-module $M$ as usual: The character
is a map $\ch M:X\to \setN$ given by $\ch M(\mu) = \dim M_\mu$. Let
$e^\mu$ be the delta function $e^\mu(\gamma) =
\delta_{\mu,\gamma}$. We will write $\ch M$ as the formal infinite sum
\begin{equation*}
  \ch M = \sum_{\mu \in X} \dim M_\mu e^\mu.
\end{equation*}
For more details see e.g.~\cite{Humphreys}.  Note that if we define
$w(\sum_{\mu}a_\mu e^\mu)=\sum_{\mu} a_\mu e^{w(\mu)}$ then $\ch
{^w}M=w(\ch M)$ by the above considerations.

\begin{prop}
  \label{character-for-twisted-vermas}
  \begin{equation*}
    \ch T_w M(\lambda)= \ch M(w.\lambda)
  \end{equation*}
\end{prop}
\begin{proof}
  To determine the character of $T_wM(\lambda)$ we would like to find
  a basis. We will do this by looking at some vectorspace
  isomporphisms to a space where we can easily find a basis. Then use
  the isomorphisms back again to determine what the basis looks like
  in $T_wM(\lambda)$.  So assume $w=s_{i_r}\cdots s_{i_1}$ is a
  reduced expression for $w$. Expand to a reduced expression
  $s_{i_N}\cdots s_{i_{r+1}}s_{i_r}\cdots s_{i_1}$ for $w_0$.  Let
  $U_v^w=\spa{\setQ(v)}{F_{\beta_{r+1}}^{a_{r+1}}\cdots
    F_{\beta_N}^{a_N}|a_i\in\setN}$.  Set $k=\setQ(v)$. We have the
  canonical vector space isomorphisms
  \begin{align*}
    U_v^-(w)^* \tensor_{U_v^-(w)} U_v \tensor_{U_v} U_v
    \tensor_{U_v^{\geq 0}} k_\lambda \iso& U_v^-(w)^*
    \tensor_{U_v^-(w)} U_v \tensor_{U_v^{\geq 0}} k_\lambda
    \\
    \iso& U_v^-(w)^* \tensor_k U_v^w \tensor_k k_\lambda.
  \end{align*}
  The map from the last vectorspace to the first is easily seen to be
  $f\tensor u\tensor 1 \mapsto f\tensor u \tensor 1\tensor 1 =
  f\tensor u\tensor v_\lambda$, $f\in U_v^-(w)^*$, $u\in U_v^w$ and
  $v_\lambda=1\tensor 1\in U_v \tensor_{U_v^{\geq 0}}
  k_\lambda=M(\lambda)$ is a highest weight vector in $M(\lambda)$.

  So we see that a basis of
  $T_wM(\lambda)={^w}(U_v^-(w)^*\tensor_{U_v^-(w)}U_v\tensor_{U_v} M)$
  is given by the following: Choose a basis $\{f_i\}_{i\in I}$ for
  $U_v^-(w)^*$ and a basis $\{u_j\}_{j\in J}$ for $U_v^w$. Then a
  basis for $T_wM(\lambda)$ is given by
  \begin{equation*}
    \{f_i\tensor u_j\tensor v_\lambda\}_{i\in I, j\in J}.
  \end{equation*}

  So we can find the weights of $T_wM(\lambda)$ by examining the
  weights of $f\tensor u\tensor v_\lambda$ for $f\in U_v^-(w)^*$ and
  $u\in U_v^w$. By the remarks before this proposition we have that
  $K_\alpha(f\tensor 1)=v^{(\gamma+\sum_{i=1}^r
    \beta_i|\alpha)}f\tensor K_\alpha$ for $f\in
  U_v^-(w)^*_{v^{\gamma}}$ so for such $f$ and for $u\in
  (U_v^w)_{v^\mu}$ the weight of $f\tensor u\tensor v_\lambda$ is
  $v^{\gamma+\mu+\sum_{i=1}^r\beta_i}\lambda$. After the twist with
  $w$ the weight is $v^{w(\gamma+\mu)}w.\lambda$. The weights $\gamma$
  and $\mu$ are exactly such that $w(\gamma)<0$ and $w(\mu)<0$ so we
  see that the weights of $T_wM(\lambda)$ are $\{v^\mu
  w.\lambda|\mu<0\}$ each with multiplicity ${\cal P}(\mu)$ where
  ${\cal P}$ is Kostant's partition function. This proves that the
  character is the same as the character for the Verma module
  $M(w.\lambda)$.
\end{proof}

\begin{defn}
  Let $\lambda\in X$ and $M(\lambda)$ the Verma module with highest
  weight $\lambda$. Let $w\in W$. We define
  \begin{equation*}
    M^w(\lambda) = T_wM(w\inv.\lambda).
  \end{equation*}
\end{defn}

Recall the duality functor $D:U_v-\Mod \to U_v-\Mod$. For a $U_v$
module $M$, $DM=\operatorname{Hom}(M,\setQ(v))$ is the graded dual module with
action given by $(xf)(m)=f(S(\omega(m)))$ for $x\in U_v$, $f\in DM$
and $m\in M$. By this definition we have $\ch DM = \ch M$ and $D(DM) =
M$.

\begin{thm}
  \label{sec:twisting-functors-9}
  Let $w_0$ be the longest element in the Weyl group. Let $\lambda \in
  X$. Then
  \begin{equation*}
    T_{w_0}M(\lambda)\iso DM(w_0.\lambda)
  \end{equation*}
\end{thm}
\begin{proof}
  We will show that $DT_{w_0}M(w_0.\lambda)\iso M(\lambda)$ by showing
  that $DT_{w_0}M(w_0.\lambda)$ is a highest weight module with
  highest weight $\lambda$. We already know that the characters are
  equal by Proposition~\ref{character-for-twisted-vermas} so all we
  need to show is that $DT_{w_0}M(w_0.\lambda)$ has a highest weight
  vector of weight $\lambda$ that generates the whole module over
  $U_v$. Consider the function $g_\lambda\in DM^{w_0}(\lambda)$ given
  by:
  \begin{equation*}
    g_\lambda(F_{\beta_N}^{-a_N-1}\tensor \dots \tensor F_{\beta_1}^{-a_1-1}\tensor v_{w_0.\lambda})=
    \begin{cases}
      1 &\text{ if } a_N=\dots=a_1=0
      \\
      0 &\text{ otherwise}.
    \end{cases}
  \end{equation*}
  We claim that $F_{\beta_N}^{-a_N-1}\tensor \dots \tensor
  F_{\beta_1}^{-a_1-1}\tensor v_{w_0.\lambda}$ with $a_i\in \setN$
  defines a basis for $M^{w_0}(\lambda)$ so this defines a function on
  $M^{w_0}(\lambda)$.  In the proof of
  Proposition~\ref{character-for-twisted-vermas} we see that a basis
  is given by $f\tensor 1\tensor v_\lambda\in U_v^-(w_0)\tensor
  U_v\tensor M(\lambda)=T_{w_0}M(\lambda)$. We know that elements of
  the form $f_{m_N}^{(N)}\cdots f_{m_1}^{(1)}$ defines a basis of
  $(U_v^-)^*=U_v^-(w_0)^*$.  Under the isomorphisms of
  Proposition~\ref{prop:7} $f_{m_N}^{(N)}\cdots f_{m_1}^{(1)}\tensor
  1\tensor v_{w_0.\lambda}$ is sent to
  \begin{equation*}
    K_{\beta_N}F_{\beta_{N}}^{-m_{N}-1}\tensor \cdots \tensor K_{\beta_1}F_{\beta_1}^{-m_1-1}\tensor v_{w_0.\lambda}\in S_v(F_{\beta_N})\tensor_{U_v}\cdots \tensor_{U_v} S_v(F_{\beta_1})\tensor_{U_v} M(w_0.\lambda).
  \end{equation*}
  If we commute all the $K$'s to the right to the $v_\lambda$ we get
  some non-zero multiple of
  \begin{equation*}
    F_{\beta_{N}}^{-m_{N}-1}\tensor \cdots \tensor F_{\beta_1}^{-m_1-1}\tensor v_{w_0.\lambda}.
  \end{equation*}
  So we have shown that $\{F_{\beta_{N}}^{-m_{N}-1}\tensor \cdots
  \tensor F_{\beta_1}^{-m_1-1}\tensor v_{w_0.\lambda} |m_i\in\setN\}$
  is a basis of $M^{w_0}(\lambda)$.

  The action on a dual module $DM$ is given by
  $uf(u')=f(S(\omega(u)u'))$. Remember that the action on
  $M^{w_0}(\lambda)$ is twisted by $R_{w_0}$ so we get that
  \begin{equation*}
    ug_\lambda(F_{\beta_N}^{-a_N-1}\tensor \dots \tensor F_{\beta_1}^{-a_1-1}\tensor v_{w_0.\lambda})= g_\lambda(R_{w_0}(S(\omega(u)))F_{\beta_N}^{-a_N-1}\tensor \dots \tensor F_{\beta_1}^{-a_1-1}\tensor v_{w_0.\lambda}).
  \end{equation*}
  In particular for $u=K_\mu$ we get
  \begin{align*}
    K_\mu g_\lambda(F_{\beta_N}^{-a_N-1}\tensor \dots \tensor
    F_{\beta_1}^{-a_1-1}\tensor v_{w_0.\lambda}) &=
    g_\lambda(K_{w_0(\mu)}F_{\beta_N}^{-a_N-1}\tensor \dots \tensor
    F_{\beta_1}^{-a_1-1}\tensor v_{w_0.\lambda})
    \\
    &=
    v^{c}(w_0.\lambda)(K_{w_0(\mu)})g_\lambda(F_{\beta_N}^{-a_N-1}\tensor
    \dots \tensor F_{\beta_1}^{-a_1-1}\tensor v_{w_0.\lambda})
  \end{align*}
  where
  \begin{align*}
    c&=(w_0(\mu)|\sum_{i=1}^N a_i\beta_i + \sum_{i=1}^N \beta_i).
  \end{align*}

  we have
  \begin{align*}
    v^c (w_0.\lambda)(K_{w_0(\mu)}) =& v^{(w_0(\mu)|\sum_{i=1}^N
      a_i\beta_i + \sum_{i=1}^N \beta_i)} \left(v^{-\rho}w_0(v^\rho
      \lambda)\right)(K_{w_0(\mu)})
    \\
    =& v^{(w_0(\mu)|\sum_{i=1}^N a_i\beta_i + 2\rho)}
    v^{-(\rho|w_0(\mu))}(v^\rho\lambda)(K_\mu)
    \\
    =& v^{(w_0(\mu)|\sum_{i=1}^N a_i\beta_i+\rho)}
    v^{(\rho|\mu)}\lambda(K_\mu)
    \\
    =& v^{(w_0(\mu)|\sum_{i=1}^N a_i\beta_i+\rho)}
    v^{-(\rho|w_0(\mu))}\lambda(K_\mu)
    \\
    =& v^{(w_0(\mu)|\sum_{i=1}^N a_i\beta_i)}\lambda(K_\mu).
  \end{align*}
  Setting the $a_i$'s equal to zero we get $\lambda(K_\mu)$.  So
  $g_\lambda$ has weight $\lambda$. We want to show that $g_\lambda$
  generates $DM^{w_0}(\lambda)$ over $U_v$.
  
  Let $M\in\setN^N$, $M=(m_1,\dots,m_N)$. An element in
  $DM^{w_0}(\lambda)$ is a linear combination of elements of the form
  $g_M$ defined by:
  \begin{equation*}
    g_M(F_{\beta_N}^{-a_N-1}\tensor \dots \tensor F_{\beta_1}^{-a_1-1}\tensor v_{w_0.\lambda})=\delta_{a_1,m_1}\cdots \delta_{a_N,m_N}.
  \end{equation*}
  This is because of the way the dual module is defined (as the graded
  dual).  We want to show that $g_M\in U_v g_\lambda$ by using
  induction over $m_1+\cdots m_N$.  Note that
  $g_{(0,\dots,0)}=g_\lambda$ so this gives the induction start.
  Assume $M=(m_1,\dots,m_N)\in\setN^N$. Let $j$ be such that
  $m_N=\dots=m_{j+1}=0$ and $m_j>0$. By induction we get for
  $M'=(0,\dots,0,m_j-1,m_{j-1},\dots,m_1)$ that $g_{M'}\in U_v
  g_\lambda$. Now let $u_{j}=\omega(S\inv (R_{w_0}\inv
  (F_{\beta_j})))$. Then
  \begin{align*}
    u_jg_\lambda(F_{\beta_N}^{-a_N-1}\tensor \dots \tensor
    F_{\beta_1}^{-a_1-1}\tensor v_{w_0.\lambda}) &=
    g_\lambda(F_{\beta_j}F_{\beta_N}^{-a_N-1}\tensor \dots \tensor
    F_{\beta_1}^{-a_1-1}\tensor v_{w_0.\lambda}).
  \end{align*}
  From Lemma~\ref{sec:twisting-functors-6} we get for $r>j$ (setting
  $k=\left< \beta_j,\beta_r^\vee\right>$)
  \begin{equation*}
    F_{\beta_j}F_{\beta_r}^{-a} = v_{\beta_r}^{-ak}F_{\beta_r}^{-a} + \sum_{i\geq 1} v_{\beta_r}^{-ak - (a+1)i}{a+i-1\brack i}_{\beta_r}  F_{\beta_r}^{-i-a}\tilde{\ad}(F_{\beta_r}^i)(u).
  \end{equation*}
  But $g_{M'}$ is zero on every $F_{\beta_N}^{-a_N-1}\tensor \dots
  \tensor F_{\beta_1}^{-a_1-1}\tensor v_{w_0.\lambda}$ where one of the
  $a_i$'s with $i>j$ is strictly greater than zero. This coupled with
  the observation above gives us that
  \begin{align*}
    &u_jg_{M'}(F_{\beta_N}^{-a_N-1}\tensor \dots \tensor
    F_{\beta_1}^{-a_1-1}\tensor v_{w_0.\lambda})
    \\
    =& g_{M'}(v^c F_{\beta_N}^{-a_N-1}\tensor \dots\tensor
    F_{\beta_j}^{-(a_j-1)-1} \tensor\dots \tensor
    F_{\beta_1}^{-a_1-1}\tensor v_{w_0.\lambda})
    \\
    =&v^c g_{M}(F_{\beta_N}^{-a_N-1}\tensor \dots\tensor
    F_{\beta_j}^{-a_j-1} \tensor \cdots \tensor
    F_{\beta_1}^{-a_1-1}\tensor v_{w_0.\lambda})
  \end{align*}
  where $c$ is some constant coming from the commutations.  We see
  that $g_M=v^{-c}u_jg_{M'}$ which finishes the induction step.
  
  So in conclusion we have that $DM^{w_0}(\lambda)$ is a highest
  weight module with highest weight $\lambda$. So we have a surjection
  from $M(\lambda)$ to $DM^{w_0}(\lambda)$. But since the two modules
  have the same character and the weight spaces are finite dimensional
  the surjection must be an isomorphism.
\end{proof}

\begin{prop}
  \label{prop:8}
  Let $M$ be a $U_v$-module, $\beta\in\Phi^+$ and let $w\in W$. Assume
  $s_{i_r}\cdots s_{i_1}$ is a reduced expression of $w$ and
  $F_{\beta}=R_{s_{i_1}}\cdots R_{s_{i_r}}(F_{\alpha})$ for some
  $\alpha \in \Pi$ such that $l(s_\alpha w)>l(w)$ (so we have
  $w(\beta)=\alpha$). Then
  \begin{equation*}
    {^w}(S_v(F_{\beta})\tensor_{U_v} M) \iso S_v(F_\alpha)\tensor_{U_v} {^w}M.
  \end{equation*}
\end{prop}
\begin{proof}
  Define the map $\phi:S_v(F_\alpha)\tensor {^w}M\to
  {^w}(S_v(F_{\beta})\tensor M) $ by
  \begin{equation*}
    \phi(uF_{\alpha}^{-m}\tensor m)=R_{w\inv}(u)F_{\beta}^{-m}\tensor m.
  \end{equation*}
  This is obivously a $U_v$-homomorphism if it is welldefined and it
  is a bijection because $R_{w\inv}$ is a $U_v$-isomorphism. We have
  to check that if $uF_{\alpha}^{-m}=u'F_{\alpha}^{-m'}$ then
  $R_{w\inv}(u)F_{\beta}^{-m}=R_{w\inv}(u')F_\alpha^{-m'}$ and that
  $\phi(uF_\alpha^{-m} u'\tensor m )=\phi(uF_{\alpha}^{-m} \tensor
  R_{w\inv}(u')m)$ but $uF_\alpha^{-m}=u'F_{\alpha}^{-m'}$ if and only
  if $F_\alpha^{m'}u=F_\alpha^{m}u'$. Using the isomorhpism
  $R_{w\inv}$ on this we get
  $F_\beta^{m'}R_{w\inv}(u)=F_\beta^{m}R_{w\inv}(u')$ which implies
  $R_{w\inv}(u)F_\beta^{-m}=R_{w\inv}(u')F_{\beta}^{-m'}$. For the
  other equation: Since we only have the definition of $\phi$ on
  elements on the form $uF_{\alpha}^{-m}\tensor m$ assume
  $F_\alpha^{-m}u'=\tilde{u}F_{\beta}^{-\tilde{m}}$. This is
  equivalent to $u'F_\alpha^{\tilde{m}}=F_\alpha^{m}\tilde{u}$. Use
  $R_{w\inv}$ on this to get
  $R_{w\inv}(u')F_\beta^{\tilde{m}}=F_\beta^{m}\tilde{u}$ or
  equivalently
  $F_\beta^{-m}R_{w\inv}(u)=R_{w\inv}(\tilde{u})F_\alpha^{-\tilde{m}}$. Now
  we can calculate:
  \begin{align*}
    \phi(uF_\alpha^{-m} u'\tensor
    m)=&\phi(u\tilde{u}F_\alpha^{-\tilde{m}}\tensor m)
    \\
    =& R_{w\inv}(u\tilde{u})F_\beta^{-\tilde{m}}\tensor m
    \\
    =& R_{w\inv}(u)R_{w\inv}(\tilde{u})F_\beta^{-\tilde{m}}\tensor m
    \\
    =& R_{w\inv}(u)F_\beta^{-m}\tensor
    R_{w\inv}(u')m=\phi(uF_\alpha^{-m}\tensor R_{w\inv}(u)m).
  \end{align*}
\end{proof}

\begin{prop}
  \label{sec:twisting-functors-8}
  $w\in W$.  If $s$ is a simple reflection such that $sw>w$ then
  \begin{equation*}
    T_{sw}=T_s\circ T_w.
  \end{equation*}
\end{prop}
\begin{proof}
  Let $\alpha$ be the simple root corresponding to the simple
  reflection $s$. By Proposition~\ref{sec:twisting-functors-4} we get
  for $M$ a $U_v$-module:
  \begin{align*}
    T_{sw}M= {^{sw}}(S_v^{sw}\tensor_{U_v} M) \iso&
    {^{sw}}(S_v(R_{w\inv}(F_{\alpha}))\tensor_{U_v} S_v^w
    \tensor_{U_v} M)
    \\
    \iso& {^s} (^{w}(S_v(R_{w\inv}(F_{\alpha}))\tensor_{U_v} S_v^w
    \tensor_{U_v} M))
    \\
    \iso& {^s}(S_v(F_\alpha)\tensor_{U_v} {^{w}}(S_v^w \tensor_{U_v}
    M))
  \end{align*}
  where the last isomorphism is the one from Proposition~\ref{prop:8}.
\end{proof}

\section {Twisting functors over Lusztigs A-form}
\label{sec:twist-funct-over}
We want to define twisting functors so they make sense to apply to
$U_A$ modules. Note first that the maps $R_s$ send $U_A$ to $U_A$.

Recall that for $n\in \setN$ with $n>0$ and $F_{\beta}$ a root vector
we have defined in $U_{v(F_\beta)}$
\begin{equation}
  \label{eq:inverse_divided_powers}
  F_{\beta}^{(-n)}= [n]_{\beta}! F_{\beta}^{-n}
\end{equation}
i.e. $F_{\beta}^{(-n)}=\left(F_{\beta}^{(n)}\right)\inv$.

\begin{defn}
  Let $s$ be a simple reflection corresponding to a simple root
  $\alpha$. Let $S_A^s$ be the $U_A$-sub-bimodule of
  $S_v^s=S_v(F_\alpha)$ generated by the elements
  $\{F_\alpha^{(-n)}F_\alpha\inv |n\in \setN\}$.
\end{defn}
Note that $S_A^s\tensor_A \setQ(v)=S_v^s$.

\begin{prop}
  \label{prop:9}
  In $U_v(\frak{sl}_2)$ let $E,K,F$ be the usual generators and define
  as in~\cite{MR1066560} the elements
  \begin{equation*}
    { K;c \brack t}=\prod_{n=1}^t \frac{Kv^{c-n+1}-K\inv v^{-c+n-1}}{v^s-v^{-s}}.
  \end{equation*}
  Then
  \begin{equation*}
    F^{(-s)}F\inv E^{(r)}=\sum_{t=0}^r E^{(r-t)}{K;r-s-t-2 \brack t} F^{(-s-t)}F\inv.
  \end{equation*}
\end{prop}
\begin{proof}
  This is proved by induction over $r$. We define as in~\cite{Jantzen}
  \begin{equation*}
    [K;c]={K;c\brack 1} = \frac{Kv^c-K\inv v^{-c}}{v-v\inv}.
  \end{equation*}
  From~\cite{Jantzen} we get $EF^{s+1}=F^{s+1} E+[s+1]F^{s}[K,-s]$ so
  \begin{equation*}
    F^{-s-1}E=EF^{-s-1}+[s+1]F\inv [K;-s]F^{-s-1}=EF^{-s-1}+[s+1] [K;-2-s]F^{-s-2}
  \end{equation*}
  and multiplying with $[s]!$ we get
  \begin{equation*}
    F^{(-s)}F\inv E=EF^{(-s)}F\inv+ [K;-2-s]F^{(-s-1)}F\inv .
  \end{equation*}
  This is the induction start. The rest is the induction step. In the
  process you have to use that
  \begin{equation*}
    \frac{1}{[r]}\left([r-t]{K;r-s-t\brack t} + {K;r-1-s-t\brack t-1}[K;-s-t]\right)={K;r-s-t-1\brack t}
  \end{equation*}
  or equivalently that
  \begin{equation*}
    [r-t][K;r-s-t]+[t][K;-s-t]=[r][K;r-s-2t].
  \end{equation*}
  This can be shown by a direct calculation.
\end{proof}

We could have proved this in the other way around instead too to get
\begin{prop}
  \label{prop:2}
  \begin{equation*}
    E^{(r)}F^{(-s)}F\inv =\sum_{t=0}^r F^{(-s-t)}F\inv {K;s+t-r+2 \brack t} E^{(r-t)}.
  \end{equation*}
\end{prop}

The above and Corollary~\ref{cor:3} shows that $S_A(F)$ is a
bimodule. We can now define the twisting functor $T_s^A$ corresponding
to $s$:
\begin{defn}
  Let $s$ be a simple reflection corresponding to a simple root
  $\alpha$. The twisting functor $T_s^A:U_A\operatorname{-Mod}\to
  U_A\operatorname{-Mod}$ is defined by: Let $M$ be a $U_A$ module,
  then
  \begin{equation*}
    T_s^A(M)={^s}(S_A(F_\alpha) \tensor_{U_A} M).
  \end{equation*}
\end{defn}
Note that $T_s^A(M)\tensor_A \setQ(v)=T_s(M\tensor_A \setQ(v))$ so
that if $M$ is a $\setQ(v)$ module then $T_s^A=T_s$ on $M$.

We want to define the twisting functor for every $w\in W$ such that if
$w$ has a reduced expression $w=s_{i_r}\cdots s_{i_1}$ then
$T_w^A=T_{s_{i_r}}^a\circ \cdots \circ T_{s_{i_1}}^A$. As before we
define a 'semiregular bimodule'
$S_A^w=U_A\tensor_{U_A^-(w)}U_A^-(w)^*$ and show this is a bimodule
isomorphic to $S_A(F_{\beta_r})\tensor_{U_A}\cdots \tensor_{U_A}
S_A(F_{\beta_1})$.

\begin{thm}
  \label{thm:S_bimodule}
  $S_A^w:=U_A\tensor_{U_A^-(w)} U_A^-(w)^*$ is a bimodule isomorphic
  to $S_A(F_{\beta_r})\tensor_{U_A}\cdots \tensor_{U_A}
  S_A(F_{\beta_1})$ and the functors $T_s^A$, $s\in \Pi$ satisfy braid
  relations.
\end{thm}
\begin{proof}
  Note that $U_A^-(w)$ can be seen as an $A$-submodule of $U_v^-(w)$
  and similarly $U_A^-(w)^*$ can be seen as a submodule of
  $U_v^-(w)^*$. So we have an injective $A$ homomorphism
  \begin{equation*}
    S_A^w\to S_v^w.
  \end{equation*}
  Assume the length of $w$ is $r$ and $w=s_{i_r}w'$, $l(w')=r-1$.  We
  want to show that the isomorphism $\phi_r$ from
  Proposition~\ref{sec:twisting-functors-4} restricts to an
  isomorphism $S_A^w\to S_A(F_{\beta_r})\tensor_{U_A} S_A^{w'}$.

  Assume $f\in U_A^-(w)$ is such that $f=g\cdot f'_m$ meaning that
  $f(xF_{\beta_r}^{(n)})=g(x)\delta_{m,n}$, ($x\in U_A^-(w')$,
  $n\in\setN$) where $g\in U_A^-(w')^*$. Then $f'_m=[m]_{\beta_r}!f_m$
  where $f_m$ is defined like in
  Proposition~\ref{sec:twisting-functors-4} and for $u\in U_A$ we have
  therefore
  \begin{equation*}
    \phi_r(u\tensor f)= uF_{\beta_r}^{(-m)}F_{\beta_r}\inv \tensor (1\tensor g)
  \end{equation*}
  which can be seen to lie in $S_A(F_{\beta_r})\tensor_{U_A}
  S_A^{w'}$. The inverse also restricts to a map to the right space:
  \begin{align*}
    \psi_r(uF_{\beta_r}^{(-m)}F_{\beta_r}\inv \tensor (1\tensor
    g))=&\psi_r(u[m]_{\beta_r}!F_{\beta_r}^{-m-1}\tensor (1\tensor g))
    \\
    =& [m]_{\beta_r}! u\tensor f_{m}\cdot g
    \\
    =& u\tensor f'_{m}\cdot g.
  \end{align*}
  The maps are well defined because they are restrictions of well
  defined maps and it is easy to see that they are inverse to each
  other.

  As in the generic case we get a right module action on $S_A^w$ in
  this way. This is the right action coming from $S_v^w$ restricted to
  $S_A^w$.  So now we have $S_A^w=S_A(F_{\beta_r})\tensor_{U_A}\cdots
  \tensor_{U_A} S_A(F_{\beta_1})$. Showing that the twisting functors
  then satisfy braid relations is done in the same way as in
  Proposition~\ref{sec:twisting-functors-8}.
\end{proof}

Now we can define $T_w^A=T_{s_{i_r}}^A\circ\cdots \circ T_{s_{i_1}}^A$
if $w=s_{i_r}\cdots s_{i_1}$ is a reduced expression of $w$. By the
previous theorem there is no ambiguity in this definition since the
$T_s^A$'s satisfy braid relations.

It is now possible for any $A$ algebra $R$ to define twisting functors
$U_R$-Mod$\to U_R$-Mod. Just tensor over $A$ with $R$.

F.x. let $R=\setC$ with $v\mapsto 1$. $S_A(F_\beta)\tensor_A \setC$ is
just the normal $S^s=U_{(y_\beta)}/U$ via the isomorphism
$uF_\beta^{(-n)}F_\beta\inv \tensor 1 \mapsto \bar{u} y_\beta^{-n-1}$
where $\bar{u}$ is given by the isomorphism between $U_A^-\tensor_A
\setC$ and $U^-$.

\begin{thm}
  \label{thm:A_dual}
  Let $R$ be an $A$-algebra with $v\in A$ being sent to $q\in
  R\backslash\{0\}$. Let $\lambda:U_R^0 \to R$ be an $R$-algebra
  homomorphism and let $M_R(\lambda)=U_R\tensor_{U_R^{\geq 0}}
  R_\lambda$ be the $U_R$ Verma module with highest weight $\lambda$
  where $R_\lambda$ is the rank~$1$ free $U_R^{\geq 0}$-module with
  $U_R^{>0}$ acting trivially and $U_R^0$ acting as $\lambda$. Let
  $D:U_R\to U_R$ be the duality functor on $U_R-\Mod$ induced from the
  duality functor on $U_A\to U_A$. Then
  \begin{equation*}
    T_{w_0}^R M_R(\lambda)\iso DM_R(w_0.\lambda).
  \end{equation*}
\end{thm}
\begin{proof}
  The proof is the almost the same as the proof of
  Theorem~\ref{sec:twisting-functors-9}.  We have by
  Corollary~\ref{cor:3} (setting
  $k=\left<\beta_j,\beta_r^\vee\right>$)
  \begin{equation*}
    F_{\beta_j} F_{\beta_r}^{(-a)}F_{\beta_r}\inv = q^{-(a+1)(\beta_r|\beta_j)}F_{\beta_r}^{(-a)}F_{\beta_r}\inv F_{\beta_j} + \sum_{j\geq 1} q_{\beta_r}^{-(a+1)k-(a+2)i}F_{\beta_r}^{(-a-i)}F_{\beta_r}\inv \tilde{\ad}(F_{\beta_r}^{(i)})(u).
  \end{equation*}
  
  Define for $M=(m_1,\dots,m_N)\in\setN$ the function
  \begin{equation*}
    g_M(F_{\beta_N}^{(-a_N)}F_{\beta_N}\inv \tensor \cdots \tensor F_{\beta_1}^{(-a_1)}F_{\beta_1}\inv \tensor v_{w_0.\lambda})=
    \begin{cases}
      1 &\text{ if } a_1=m_1 \dots a_N=m_N
      \\
      0 \text{ otherwise }.
    \end{cases}
  \end{equation*}
  Note that $g_{(0,\dots,0)}=g_\lambda$ from
  Theorem~\ref{sec:twisting-functors-9}. In particular it has weight
  $\lambda$. We want to show that $DM_R^{w_0}(\lambda)=U_R
  g_{(0,\dots,0)}$. We use induction on the number of nonzero entries
  in $M$. Assume $j$ is such that $m_N=\dots=m_{j+1}=0$ and
  $m_{j}=n>0$. Let $M'=(0,\dots,0,0,m_{j-1},\dots,m_1)$. By induction
  $g_{M'}\in U_R g_{(0,\dots,0)}$.
  
  Set $u=\omega(S\inv(R_{w_0}\inv(F_{\beta_j}^{(n)})))$. Then
  \begin{align*}
    &ug_{M'}( F_{\beta_N}^{(-a_N)}F_{\beta_N}\inv \tensor \cdots
    \tensor F_{\beta_1}^{(-a_1)}F_{\beta_1}\inv \tensor
    v_{w_0.\lambda})
    \\
    =&g_{M'}(F_{\beta_j}^{(n)} F_{\beta_N}^{(-a_N)}F_{\beta_N}\inv
    \tensor \cdots \tensor F_{\beta_1}^{(-a_1)}F_{\beta_1}\inv \tensor
    v_{w_0.\lambda})
    \\
    =&g_{M'}(\frac{1}{[n]_{\beta_j}!}F_{\beta_j}^{n}
    F_{\beta_N}^{(-a_N)}F_{\beta_N}\inv \tensor \cdots \tensor
    F_{\beta_1}^{(-a_1)}F_{\beta_1}\inv \tensor v_{w_0.\lambda})
    \\
    =& g_{M'}(q^{c_1}
    \frac{1}{[n]_{\beta_j}!}F_{\beta_j}^{n-1}F_{\beta_N}^{(-a_N)}F_{\beta_N}\inv\tensor
    \cdots \tensor F_{\beta_j}F_{\beta_j}^{(-a_j)}F_{\beta_j}\inv
    \tensor \cdots \tensor F_{\beta_1}^{(-a_1)}F_{\beta_1}\inv \tensor
    v_{w_0.\lambda})
    \\
    \vdots
    \\
    =&g_{M'}(q^{c_n}\frac{1}{[n]_{\beta_j}!}
    F_{\beta_N}^{(-a_N)}F_{\beta_N}\inv \tensor \cdots\tensor
    F_{\beta_j}^{n}F_{\beta_j}^{(-a_j)}F_{\beta_j}\inv \tensor\cdots
    \tensor F_{\beta_1}^{(-a_1)}F_{\beta_1}\inv \tensor
    v_{w_0.\lambda})
    \\
    =&\begin{cases} g_{M'}(q^{c_n}F_{\beta_N}^{(-a_N)}F_{\beta_N}\inv
      \tensor \cdots\tensor {a_j\brack
        n}_{\beta_j}F_{\beta_j}^{(-(a_j-n))}F_{\beta_j}\inv
      \tensor\cdots \tensor F_{\beta_1}^{(-a_1)}F_{\beta_1}\inv
      \tensor v_{w_0.\lambda}) &\text{ if } n\leq a_j
      \\
      0 &\text{ otherwise }
    \end{cases}
  \end{align*}
  for some appropiate integers $c_1,\dots,c_n\in \setZ$. $g_{M'}$ is
  nonzero on this only when $n=a_j$. So we get in conclusion that
  $ug_{M'}=v^{-c_n}g_{M}$. This finishes the induction step.
\end{proof}

\section{$\mathfrak{sl}_2$ calculations}
\label{sec:mathfr-calc}
Assume $\mathfrak{g}=\mathfrak{sl}_2$. Let $r\in \setN$. Let
$M_A(v^r)$ be the $U_A(\mathfrak{sl}_2)$ Verma module with highest
weight $v^r\in \setZ$ i.e. $M_A(v^r)=U_A\tensor_{U_A^{\geq 0}}
A_{v^r}$ where $A_{v^r}$ is the free $U_A^{\geq 0}$-module of rank $1$
with $U_A^{+}$ acting trivially and $K\cdot 1 = q^r$.  Inspired
by~\cite{HHA-kvante} we see that in $\mathfrak{sl}_2$ we have for $r
\in \setZ$ the homomorphism $\phi:M_A(v^r)\to DM_A(v^r)$ given by:

Let $\{w_i=F^{(i)}w_0\}$ be a basis for $M_A(\lambda)$ where $w_0$ is
a highest weight vector in $M_A(v^r)$ and let $\{w_i^*\}$ be the dual
basis in $DM_A(\lambda)$. Then
\begin{equation*}
  \phi(w_i)=(-1)^iv^{i(i-1)-ir}{r\brack i}w_i^*.
\end{equation*}
Checking that this is indeed a homomorphism of $U_A$ algebras is a
straightforward calculation.

By Theorem~\ref{thm:A_dual} we see that $DM_A(v^r)=M_A^s(v^r)$. In the
following section we will try to say something about the composition
factors of a Verma module so it is natural to consider first
$\mathfrak{sl}_2$ Verma modules.

\begin{defn}
  Let $\mathfrak{g}=\mathfrak{sl}_2$. Let $r\in \setN$. Then
  $H_A(v^r)$ is defined to be the free $U_A(\mathfrak{sl}_2)$-module
  of rank $r+1$ with basis $e_0,\dots,e_r$ defined as follows:
  \begin{align*}
    K e_i =& v^{r-2i}e_i, {K;c\brack t}e_i ={r-2i+c \brack t} e_i
    \\
    E^{(n)}e_i =& {i \brack n} e_{i-n},\quad n\in \setN
    \\
    F^{(n)}e_i =& {r-i\brack n} e_{i+n},\quad n\in \setN
  \end{align*}
  for $i=0,\dots, r$. Where $e_{<0}=0=e_{>r}$.
\end{defn}

\begin{lemma}
  \label{lemma:2}
  Let $\mathfrak{g}=\mathfrak{sl}_2$. Let $r\in \setN$. Then we have a
  short exact sequence:
  \begin{equation*}
    0\to DM_A(v^{-r-2}) \to M_A(v^r) \to H_A(v^r) \to 0.
  \end{equation*}
\end{lemma}
\begin{proof}
  We use the fact that $DM_A(v^{-r-2})=T_s^A M_A(v^r)$ by
  Theorem~\ref{thm:A_dual}. Let $e_i=F^{(i)}w_0$ where $w_0$ is a
  heighest weight vector in $M_A(v^r)$.  We will construct a
  $U_A$-homomorphism $ \spa{A}{e_i|i>r}\to DM_A(-r-2)$. Let $\tau$ be
  as defined in~\cite{Jantzen} Chapter~4. Note that in $U_{A(F)}$
  $S(\tau(F))$ is invertible so we can consider $S$ and $\tau$ as
  automorphisms of $U_{A(F)}$. We define a map by
  \begin{equation*}
    e_{r+i}\mapsto (-1)^{r+i}S(\tau(F^{(-i-1)}))w_0
  \end{equation*}
  Note that for $\mathfrak{sl}_2$ $R_s=S\circ \tau\circ \omega$. Using
  this and the formula in Proposition~\ref{prop:9} it is
  straightforward to check that this is a $U_A$-homomorphism.
\end{proof}
If we specialize to an $A$-algebra $R$ with $R$ being a field where
$v$ is sent to a non-root of unity $q\in R$ we get that $M_R(q^k)=U_R
\tensor_{U_A} M_A(v^k)$ is simple for $k<0$. So in the above with
$r\in\setN$, $DM_R(q^{-r-2})=M_R(q^{-r-2})=L_R(q^{-r-2})$ and actually
we see also that $H_R(q^r)=L_R(q^r)$. So there is an exact sequence
\begin{equation*}
  0\to L_R(q^{-r-2}) \to M_R(q^r) \to L_R(q^r) \to 0.
\end{equation*}
So the composition factors in $M_R(q^r)$ are $L_R(q^r)$ and
$L_R(q^{-r-2})=L_R(s.q^r)$ where $s$ is the simple reflection in the
Weyl group of $\mathfrak{sl}_2$.

\section{Jantzen filtration}
\label{sec:jantzen-filtration}

In this section we will work with the field $\setC$ and send $v$ to a
non root of unity $q\in \setC^*$. We define $U_q = U_A \tensor_A
\setC_q$ where $\setC_q$ is the $A$-algebra $\setC$ with $v$ being sent
to $q$. These results compare to the results in~\cite{HHA-kvante}
and~\cite{AL-twisted-Vermas}.

Let $\lambda$ be a weight i.e. an algebra homomorphism $U_q^0\to
\setC$ and let $M(\lambda) = U_q\tensor_{U_q^{\geq 0}}\setC_\lambda$
be the Verma module of highest weight $\lambda$. Consider the local
ring $B=\setC[X]_{(X-1)}$ and the quantum group $U_B = U_A \tensor_A
B$. We define $\lambda X:U_q^0 \to B$ to be the weight defined by
$(\lambda X)(K_\mu) = \lambda(K_\mu) X$ and we define $M_B(\lambda X)
= U_B \tensor_{U_B^{\geq 0}} B_{\lambda X}$ to be the Verma module
with highest weight $\lambda X$. Note that $M_B(\lambda X) \tensor_{B}
\setC\iso M(\lambda)$ when we consider $\setC$ as a $B$-algebra via
the specialization $X\mapsto 1$

For a simple root $\alpha_i\in \Pi$ we define $M_{B,i}(\lambda
X):=U_B(i)\tensor_{U_B^{\geq 0}}B_{\lambda X}$, where $U_B(i)$ is the
subalgebra generated by $U_B^{\geq 0}$ and $F_{\alpha_i}$. We define
$M_{B,i}^{s_i}(\lambda):=
{^{s_i}}((U_B(i)\tensor_{U_B(s_i)}U_B(s_i)^*)\tensor_{U_B(i)}M_{B,i}(s_i.\lambda))$
where the module $(U_B(i)\tensor_{U_B(s_i)}U_B(s_i)^*)$ is a
$U_B(i)$-bimodule isomorphic to
$S_{B,i}(F_{\alpha_i})=(U_{B}(i))_{(F_{\alpha_i})}/U_{B}(i)$ by similar
arguments as earlier.

\begin{prop}
  \label{prop:4}
  There exists a nonzero homomorphism $\phi:M_B(\lambda X) \to
  M_B^{s_\alpha}(\lambda X)$ which is an isomorphism if $q^\rho
  \lambda(K_\alpha) \not \in \pm q_\alpha^{\setZ_{>0}}$ and otherwise
  we have a short exact sequence
  \begin{equation*}
    0 \to M_B(\lambda X) \stackrel{\phi}{\to} M_B^{s_\alpha}(\lambda X) \to M(s_\alpha.\lambda) \to 0
  \end{equation*}
  where we have identified the cokernel
  $M_B^{s_\alpha}(s_\alpha.\lambda X)/(X-1)M_B( s_\alpha.\lambda X)$
  with $M(s_\alpha.\lambda)$.

  Furthermore there exists a nonzero homomorphism
  $\psi:M_B^{s_\alpha}(\lambda X)\to M_B(\lambda X)$ which is an
  isomorphism if $q^\rho \lambda(K_\alpha) \not \in \pm
  q_\alpha^{\setZ_{>0}}$ and otherwise we have a short exact sequence
  \begin{equation*}
    0 \to M_B^{s_\alpha}(\lambda X) \stackrel{\psi}{\to} M_B(\lambda X) \to M(\lambda)/M(s_\alpha.\lambda) \to 0.
  \end{equation*}
\end{prop}
\begin{proof}
  We will first define a map from $M_{B,i}(\lambda X)$ to
  \begin{equation*}
    M_{B,i}^{s_i}(\lambda
    X)={^{s_i}}\left((U_{B}(i))_{(F_{\alpha})}/U_{B}(i)\tensor_{U_B}
      M_{B,i}(s_\alpha.\lambda X)\right).
  \end{equation*}
  Setting $\lambda'=\lambda X$ define
  \begin{equation*}
    \phi( F_\alpha^{(n)} v_{\lambda'}) =  a_n F_{\alpha}^{(-n)}F_{\alpha}\inv \tensor v_{s_\alpha.\lambda'}
  \end{equation*}
  where
  \begin{equation*}
    a_n = (-1)^n q_\alpha^{-n(n+1)}\lambda'(K_\alpha)^{n} \prod_{t=1}^n \frac{ q_\alpha^{1-t}\lambda'(K_\alpha) - q_\alpha^{t-1} \lambda'(K_\alpha)\inv }{q_\alpha^t - q_\alpha^{-t}}.
  \end{equation*}
  So we need to check that this is a homomorphism: First of all for
  $\mu\in Q$.
  \begin{align*}
    K_\mu\cdot a_n F_\alpha^{(-n)}F_\alpha\inv \tensor
    v_{s_\alpha.\lambda'} =& a_n K_{s_{\alpha}(\mu)}
    F_\alpha^{(-n)}F_\alpha\inv \tensor v_{s_\alpha.\lambda'}
    \\
    =&
    q^{(n+1)(s_\alpha(\mu)|\alpha)}(s_\alpha.\lambda')(K_{s_\alpha(\mu)})
    F_\alpha^{(-n)}F_\alpha\inv \tensor v_{s_\alpha.\lambda'}
    \\
    =&
    q^{-(n+1)(\mu|\alpha)}q^{-(\rho|s_\alpha(\mu))}q^{(\rho|\mu)}\lambda'(K_{\mu})
    F_\alpha^{(-n)}F_\alpha\inv \tensor v_{s_\alpha.\lambda'}
    \\
    =&
    q^{-(n+1)(\mu|\alpha)}q^{-(\rho-\alpha|\mu)}q^{(\rho|\mu)}\lambda'(K_{\mu})
    F_\alpha^{(-n)}F_\alpha\inv \tensor v_{s_\alpha.\lambda'}
    \\
    =& q^{-n(\mu|\alpha)}\lambda'(K_{\mu}) F_\alpha^{(-n)}F_\alpha\inv
    \tensor v_{s_\alpha.\lambda'}
    \\
    =& \phi(K_\mu F_\alpha^{(n)}v_{\lambda'}).
  \end{align*}

  We have
  \begin{align*}
    E_\alpha\cdot a_n F_\alpha^{(-n)}F_\alpha\inv \tensor
    v_{s_\alpha.\lambda '} =& a_n R_{s_i}(E_{\alpha_i})
    F_{\alpha}^{(-n)}F_\alpha\inv \tensor v_{s_\alpha.\lambda '}
    \\
    =& - a_n F_\alpha K_\alpha F_{\alpha}^{(-n)}F_\alpha\inv \tensor
    v_{s_\alpha.\lambda '}
    \\
    =& -q_\alpha^{2(n+1)}s_\alpha.\lambda'(K_\alpha) [n]_\alpha a_n
    F_\alpha^{(-n+1)}F_\alpha\inv \tensor v_{s_\alpha.\lambda '}
    \\
    =& -q_\alpha^{2n} \lambda'(K_\alpha\inv ) [n]_\alpha a_n
    F_\alpha^{(-n+1)}F_\alpha\inv \tensor v_{s_\alpha.\lambda '}
  \end{align*}
  and
  \begin{align*}
    \phi(E_\alpha F_\alpha^{(n)}v_{\lambda '}) =& \phi\left(
      F_\alpha^{(n-1)} \frac{q_\alpha^{1-n}K_\alpha -
        q_\alpha^{n-1}K_\alpha\inv}{q_\alpha-q_\alpha\inv} v_{\lambda
        '}\right)
    \\
    =& \left( a_{n-1}F_\alpha^{(-n+1)}F_\alpha\inv
      \frac{q_\alpha^{1-n}\lambda'(K_\alpha) -
        q_\alpha^{n-1}\lambda'(K_\alpha)\inv}{q_\alpha-q_\alpha\inv}\right)\tensor
    v_{\lambda '}
  \end{align*}
  so we see that $\phi(E_\alpha F_\alpha^{(n)}v_{\lambda'}) = E_\alpha
  \cdot \phi(F_\alpha^{(n)}v_{\lambda'})$. Clearly $\phi(E_{\alpha'}
  F_\alpha^{(n)}v_{\lambda'})=0=E_{\alpha'}\cdot a_n
  F_\alpha^{(-n)}F_\alpha\inv \tensor v_{\lambda'}$ for any simple
  $\alpha'\neq \alpha$ so what we have left is $F_\alpha$: By
  Proposition~\ref{prop:2}
  \begin{align*}
    F_\alpha&\cdot a_n F_\alpha^{(-n)}F_\alpha\inv \tensor
    v_{s_\alpha.\lambda'}
    \\
    =& a_n R_{s_i}(F_\alpha) F_\alpha^{(-n)}F_\alpha\inv \tensor
    v_{s_\alpha.\lambda'}
    \\
    =& -a_n K_\alpha\inv E_{\alpha} F_\alpha^{(-n)}F_\alpha\inv
    \tensor v_{s_\alpha.\lambda'}
    \\
    =& -a_n K_\alpha\inv F_\alpha^{(-n-1)}F_\alpha\inv [K_\alpha;n+2]
    \tensor v_{s_\alpha.\lambda'}
    \\
    =& - a_n q_\alpha^{-2(n+2)} s_\alpha.\lambda'(K_\alpha\inv)
    \frac{q_\alpha^{n+2}s_\alpha.\lambda'(K_\alpha)-q_\alpha^{-n-2}s_\alpha.\lambda'(K_\alpha)\inv}{q_\alpha-q_\alpha\inv}
    F_\alpha^{(-n-1)}F_\alpha\inv \tensor v_{s_\alpha.\lambda'}
    \\
    =& - a_n q_\alpha^{-2(n+1)} \lambda'(K_\alpha)
    \frac{q_\alpha^{n}\lambda'(K_\alpha\inv)-q_\alpha^{-n}\lambda'(K_\alpha)}{q_\alpha-q_\alpha\inv}
    F_\alpha^{(-n-1)}F_\alpha\inv \tensor v_{s_\alpha.\lambda'}
  \end{align*}
  and
  \begin{align*}
    \phi(F_\alpha F_\alpha^{(n)} v_\lambda) =& [n+1]_\alpha
    \phi(F_\alpha^{(n+1)}v_\lambda)
    \\
    =& [n+1]_\alpha a_{n+1} F_\alpha^{(-n-1)}F_\alpha\inv \tensor
    v_\lambda
  \end{align*}
  so we see that $\phi(F_\alpha F_\alpha^{(n)}v_\lambda) = F_\alpha
  \cdot \phi(F_\alpha^{(n)}v_\lambda)$.

  Now note that if $\lambda(K_\alpha)\not \in \pm q_\alpha^{\setN}$
  then $X-1$ does not divide $a_n$ for any $n\in \setN$ implying that
  $a_n$ is a unit. So when $\lambda(K_\alpha)\not \in \pm
  q_\alpha^{\setN}$, $\phi$ is an isomorphism. If $\lambda(K_\alpha) =
  \epsilon q_\alpha^r$ for some $\epsilon\in\{\pm 1\}$ and $r\in
  \setN$ we see that $X-1$ divides $a_n$ for any $n>r$ so the image of
  $\phi$ is
  \begin{align*}
    \spa{B}{F_\alpha^{(-n)}F_\alpha\tensor v_{s_\alpha.\lambda'}|n\leq
      r} + (X-1)\spa{B}{F_\alpha^{(-n)}F_\alpha\inv \tensor
      v_{s_\alpha.\lambda'}|n>r}.
  \end{align*}
  Thus the cokernel $M_{B,i}^{s_i}(\lambda)/\operatorname{Im} \phi$ is
  equal to
  \begin{align*}
    \spa{B}{F_\alpha^{(-n)}F_\alpha\inv \tensor
      v_{s_\alpha.\lambda'}|n>r}/(X-1)\spa{B}{F_\alpha^{(-n)}F_\alpha\inv
      \tensor v_{s_\alpha.\lambda'}|n>r}
  \end{align*}
  which is seen to be isomorphic to
  $M_{B,i}^{s_i}(s_i.\lambda')/(X-1)M_{B,i}^{s_i}(s_i.\lambda')$.

  If $\lambda(K_\alpha)\not \in \pm q_\alpha^{\setN}$ then obviously
  we can define an inverse to $\phi$, $\psi:M_{B,i}^{s_i}(\lambda')
  \to M_{B,i}(\lambda')$. If $\lambda(K_\alpha) = \epsilon q^r$ for
  some $\epsilon\in \{\pm 1\}$ and some $r\in \setN$ we define
  $\psi:M_{B,i}^{s_i}(\lambda') \to M_{B,i}(\lambda')$ by
  \begin{equation*}
    \psi(F_\alpha^{(-n)}F_\alpha\inv \tensor v_{s_\alpha.\lambda'}) = \frac{(X-1)}{a_n} F_{\alpha}^{(n)} v_{\lambda'}
  \end{equation*}
  (note that for all $\lambda$ and all $n\in \setN$, $(X-1)^2\not |
  a_n$ so $\frac{(X-1)}{a_n}\in B$).  This implies $\phi\circ \psi =
  (X-1)\operatorname{id}$ and $\psi\circ \phi =
  (X-1)\operatorname{id}$. Using that $\phi$ is a $U_q$-homomorphism
  we show that $\psi$ is: For $u\in U_q$ and $v\in
  M_{B,i}^{s_i}(\lambda')$:
  \begin{equation*}
    (X-1) \psi(u v) = \psi( u \phi(\psi(v)))= \psi(\phi(u\psi(v)))=(X-1)u\psi(v).
  \end{equation*}
  Since $B$ is a domain this implies $\psi(uv)=u\psi(v)$.

  We see that $X-1$ divides $\frac{X-1}{a_n}$ for any $n\leq r$ so the
  image of $\psi$ is
  \begin{align*}
    (X-1)\spa{B}{F_\alpha^{(n)} v_{\lambda'}|n\leq r}+
    \spa{B}{F_\alpha^{(n)} v_{\lambda'}|n> r} .
  \end{align*}
  Thus the cokernel $M_{B,i}(\lambda)/\operatorname{Im} \psi$ is equal
  to
  \begin{align*}
    \spa{B}{F_\alpha^{(n)} v_{\lambda'}|n \leq
      r}/(X-1)\spa{B}{F_\alpha^{(n)} v_{\lambda'}|n\leq r}
  \end{align*}
  which is seen to be isomorphic to
  \begin{equation*}
    M_{B,i}(\lambda) / M_{B,i}(s_\alpha.\lambda).
  \end{equation*}

  Now we induce to the whole quantum group: We have that
  \begin{equation*}
    M_B(\lambda') = U_B \tensor_{U_B(i)} M_{B,i}(\lambda')
  \end{equation*}
  and
  \begin{align*}
    M_B^{s_i}(\lambda') =& {^{s_i}}\left( (U_B\tensor_{U_B(s_i)}
      U_B(s_i)^*)\tensor_{U_B} U_B \tensor_{U_B^{\geq 0}}
      B_{\lambda'}\right)
    \\
    \iso & {^{s_i}}\left( (U_B\tensor_{U_B(i)} U_B(i)
      \tensor_{U_B(s_i)} U_B(s_i)^*) \tensor_{U_B^{\geq 0}}
      B_{\lambda'}\right)
    \\
    \iso& U_B\tensor_{U_B(i)} {^{s_i}}\left( (U_B(i)
      \tensor_{U_B(s_i)} U_B(s_i)^*) \tensor_{U_B^{\geq 0}}
      B_{\lambda'}\right)
    \\
    \iso& U_B \tensor_{U_B(i)} M_{B,i}^{s_i}(\lambda')
  \end{align*}
  so by inducing to $U_B$-modules using the functor
  $U_B\tensor_{U_B(i)} -$ we get a map $\phi: M_B(\lambda') \to
  M_B^{s_i}(\lambda')$ and a map $\psi:M_B^{s_i}(\lambda')\to
  M_B(\lambda')$. This functor is exact on $M_{B,i}(\lambda')$ and
  $M_{B,i}^{s_i}(\lambda')$ so the proposition follows from the above
  calculations.
\end{proof}

\begin{prop}
  \label{prop:6}
  Let $\lambda:U_q^0\to \setC$ be a weight. Set $\lambda'=\lambda
  X$. Let $w\in W$ and $\alpha\in \Pi$ such that $w(\alpha)>0$. There
  exists a nonzero homomorphism $\phi:M_B^{w}(\lambda')\to
  M_B^{ws_\alpha}(\lambda')$ that is an isomorphism if $q^\rho
  \lambda(K_{w(\alpha)}) \not \in \pm q_\alpha^{\setZ_{>0}}$ and
  otherwise we have the short exact sequence
  \begin{equation*}
    0\to M_B^{w}(\lambda') \stackrel{\phi}{\to} M_B^{ws_\alpha}(\lambda') \to M^w(s_{w(\alpha)}.\lambda) \to 0
  \end{equation*}
  where the cokernel
  $M_B^{ws_\alpha}(s_{w(\alpha)}.\lambda')/(X-1)M_B^{ws_\alpha}(s_{w(\alpha)}.\lambda')$
  is identified with $M^w(s_{w(\alpha)}.\lambda)$.
  
  Furthermore there exists a nonzero homomorphism
  $\psi:M_B^{ws_\alpha}(\lambda X)\to M_B^w (\lambda X)$ which is an
  isomorphism if $q^\rho \lambda(K_{w(\alpha)}) \not \in \pm
  q_\alpha^{\setZ_{>0}}$ and otherwise we have a short exact sequence
  \begin{equation*}
    0 \to M_B^{ws_\alpha}(\lambda') \stackrel{\psi}{\to} M_B^w(\lambda') \to M^w(\lambda)/M^w(s_{w(\alpha)}.\lambda) \to 0.
  \end{equation*}
\end{prop}
\begin{proof}
  Let $\mu=w\inv.\lambda$ and $\mu'=\mu X$ then from
  Proposition~\ref{prop:4} we get a homomorphism $M_B(\mu') \to
  M_B^{s_\alpha}(\mu')$ and a homomorphism $M_B^{s_\alpha}(\mu')\to
  M_B(\mu')$. Observe that
  \begin{align*}
    q^\rho \mu(K_\alpha) =& w\inv.\lambda(K_\alpha)
    \\
    =& w\inv(q^\rho \lambda)(K_\alpha)
    \\
    =& q^{(\rho|w(\alpha))}\lambda(K_{w(\alpha)})
    \\
    =& \left(q^\rho \lambda\right)(K_{w(\alpha)})
  \end{align*}
  so $M_B(\mu') \to M_B^{s_\alpha}(\mu')$ and $M_B^{s_\alpha}(\mu')\to
  M_B(\mu')$ are isomorphisms if $\left(q^\rho
    \lambda\right)(K_{w(\alpha)}) \not \in \pm
  q_{\alpha}^{\setZ_{>0}}$ and otherwise we have the short exact
  sequences
  \begin{equation*}
    0\to M_B(\mu') \to M_B^{s_\alpha}(\mu') \to M(\mu') \to 0
  \end{equation*}
  and
  \begin{equation*}
    0 \to M_B^{s_\alpha}(\mu') \to M_B(\mu') \to M(\mu')/M(s_\alpha.\mu') \to 0.
  \end{equation*}

  Now we use the twisting functor $T_w$ on the homomorphisms
  $M_B(\mu')\to M_B^{s_\alpha}(\mu')$ and $M_B^{s_\alpha}(\mu')\to
  M_B(\mu')$ to get homomorphisms $\phi:M_B^w(\lambda) \to
  M_B^{ws_\alpha}(\lambda)$ and $\psi:M_B^{ws_\alpha}(\lambda) \to
  M_B^{w}(\lambda)$ (using the fact that $T_w\circ T_{s_\alpha}=
  T_{ws_\alpha}$). We are done if we show that $T_w$ is exact on Verma
  modules. But
  \begin{align*}
    T_w M_B(\mu') =& {^w}\left( (U_B^-(w)^* \tensor_{U_B^-(w)} U_B)
      \tensor_{U_B} U_B\tensor_{U_B^{\geq 0}} B_{\mu'} \right)
    \\
    \iso& {^w}\left( (U_B^-(w)^* \tensor_{U_B^-(w)} U_B)
      \tensor_{U_B^{\geq 0}} B_{\mu'} \right)
    \\
    \iso& {^w}\left( U_B^-(w)^* \tensor_{U_B^-(w)} U_B^-
      \tensor_{\setC} B_{\mu'} \right)
  \end{align*}
  as vectorspaces and $U_B^0$ modules. Observing that $U_B^-$ is free
  over $U_B^-(w)$ we get the exactness.
\end{proof}

Fix a weight $\lambda:U_q^0\to \setC$ and a $w\in W$. Define
$\Phi^+(w):=\Phi^+\cap w(\Phi^-) = \{\beta\in \Phi^+|w\inv(\beta)<0\}$
and $\Phi^+(\lambda) := \{\beta\in \Phi^+| q^\rho \lambda(K_\beta)\in
\pm q^\setZ\}$. Choose a reduced expression of $w_0 = s_{i_1}\cdots
s_{i_N}$ such that $w=s_{i_n}\cdots s_{i_1}$. Set
\begin{align*}
  \beta_j =
  \begin{cases}
    -ws_{i_1}\cdots s_{i_{j-1}}(\alpha_{i_j}), &\text{ if } j\leq n
    \\
    ws_{i_1}\cdots s_{i_{j-1}}(\alpha_{i_j}), &\text{ if } j> n.
  \end{cases}
\end{align*}
Then $\Phi^+=\{\beta_1,\dots,\beta_N\}$ and
$\Phi^+(w)=\{\beta_1,\dots,\beta_n\}$. We denote by
$\Psi_B^w(\lambda)$ the composite
\begin{equation*}
  M_B^w(\lambda X) \stackrel{\phi_1^w(\lambda)}{\to} M_B^{ws_{i_1}}(\lambda X) \stackrel{\phi_2^w(\lambda)}{\to} \cdots \stackrel{\phi_N^w(\lambda)}{\to} M_B^{ww_0}(\lambda X)
\end{equation*}
where the homomorphisms are the ones from Proposition~\ref{prop:6}
i.e. the first $n$ homomorphisms are the $\psi$'s and the last $N-n$
homomorphisms are the $\phi$'s from Proposition~\ref{prop:6}. We
denote by $\Psi^w(\lambda)$ the $U_q$-homomorphism $M^w(\lambda X)\to
M^{ww_0}(\lambda X)$ induced by tensoring the above $U_B$-homomorphism
with $\setC$ considered as a $B$ module by $X\mapsto 1$.

In analogy with Theorem~7.1 in~\cite{AL-twisted-Vermas} and
Proposition~4.1 in~\cite{HHA-kvante} we have
\begin{thm}
  \label{thm:Jantzen_filtration}
  Let $\lambda:U_q^0\to \setC$ be a weight. Let $w\in W$. Then there
  exists a filtration of $M^w(\lambda)$, $M^w(\lambda) \supset
  M^w(\lambda)^1 \supset \dots \supset M^w(\lambda)^r$ such that
  $M^w(\lambda)/M^w(\lambda)^1\iso \operatorname{Im}
  \Psi^w(\lambda)\subset M^{ww_0}(\lambda)$ and
  \begin{align*}
    \sum_{i=1}^r \ch M^w(\lambda)^i =& \sum_{\beta\in
      \Phi^+(\lambda)\cap \Phi^+(w)} \left( \ch M(\lambda) - \ch
      M(s_{\beta}.\lambda)\right)
    \\
    &+\sum_{\beta\in \Phi^+(\lambda)\backslash \Phi^+(w)} \ch
    M(s_{\beta}.\lambda).
  \end{align*}
\end{thm}
\begin{proof}
  Set $\lambda'=\lambda X$. Define for $i\in \setN$
  \begin{equation*}
    M_B^w(\lambda')^i = \{m\in M_B^w(\lambda') | \Psi_B^w(\lambda)(m) \in (X-1)^i M_B^{ww_0}(\lambda')\}.
  \end{equation*}
  Set $M^w(\lambda)^i = \pi(M_B^w(\lambda')^i)$ where
  $\pi:M_B^w(\lambda)\to M^w(\lambda)$ is the canonical homomorphism
  from $M_B^w(\lambda)$ to $M_B^w(\lambda)/(X-1)M_B^w(\lambda)\iso
  M^w(\lambda)$. This defines a filtration of $M^w(\lambda)$. We have
  $M^w(\lambda)^{N+1}=0$ so the filtration is finite.

  Let $\mu:U_q^0\to \setC$ be a weight. Set $\mu'=\mu X$. The maps
  $\phi_j^w(\lambda)$ restrict to weight spaces. Denote the
  restriction $\phi_j^w(\lambda)_{\mu'}$. Let
  $\Psi_B^w(\lambda)_{\mu'}:M_B^w(\lambda)_{\mu'}\to
  M_B^{ww_0}(\lambda)_{\mu'}$ be the restriction of
  $\Psi_B^w(\lambda)$ to the $\mu'$ weight space. We have a
  nondegenerate bilinear form $(-,-)$ on $M(\lambda')_{\mu'}$ given by
  $(x,y)=\left(\Psi_B^w(\lambda)_{\mu'}(x)\right)(y)$. It is
  nondegenerate since $\Psi_B^w(\lambda)$ is injective.  Let $\nu:
  B\to \setC$ be the $(X-1)$-adic valuation i.e. $\nu(b)=m$ if
  $b=(X-1)^m b'$, $(X-1)\nmid b'$. We have
  by~\cite[Lemma~5.6]{Humphreys} (originally Lemma~5.1
  in~\cite{Moduln})
  \begin{equation*}
    \sum_{j\geq 1} \dim (M_j)_\mu = \nu( \det \Psi_B^w(\lambda)_{\mu'}).
  \end{equation*}
  Clearly $\nu( \det \Psi_B^w(\lambda)_{\mu'}) = \sum_{j=1}^N \nu(\det
  \phi_j^w(\lambda)_{\mu'})$ and the result follows when we show:
  \begin{equation*}
    \nu(\det
    \phi_j^w(\lambda)_{\mu'}) = \dim_{\setC} \left( \operatorname{coker}
      \phi_j^w(\lambda)_{\mu'} \right).
  \end{equation*}
  
  Fix $\phi:= \phi_j^w(\lambda)_{\mu'}$ and let $M$ and $N$ be the
  domain and codomain respectively. $M$ and $N$ are free $B$ modules
  of finite rank. Let $d$ be the rank. We can choose bases
  $m_1,\dots,m_d$ and $n_1,\dots,n_d$ such that $\phi(m_i)=a_i n_i$,
  $i=1,\dots,d$ for some $a_i\in B$. Set $C=\operatorname{coker} \phi
  \iso \bigoplus_{i=1}^d B / (a_i)$ and set $C_\setC = C\tensor_B
  (B/(X-1)B)=C\tensor_B \setC$ where $\setC$ is considered a
  $B$-module by $X\mapsto 1$. Note that
  \begin{equation*}
    B/(a_i) \tensor_B \setC =
    \begin{cases}
      \setC, &\text{ if } (X-1) | a_i
      \\
      0, &\text{ otherwise }
    \end{cases}
  \end{equation*}
  so $\dim_{\setC} C_\setC = \# \{ i | \nu(a_i)>0\}$. Since there
  exists a $\psi:N\to M$ such that $\phi\circ \psi = (X-1)
  \operatorname{id}$ we get $\nu(a_i)\leq 1$ for all $i$. So then
  $\dim_{\setC} C_\setC = \nu(\det \phi)$ and the claim has been
  shown.
\end{proof}

\section{Linkage principle}
\label{sec:linkage-principle}

Let $R$ be a field that is an $A$-algebra and $q\in R$ the nonzero
element that $v$ is sent to.  As usual we can define the Verma
modules: Assume $\lambda:U_R^0\to R$ is a homomorphism.  Then we
define $M_R(\lambda)=U_R\tensor_{U_R^{\geq 0}} R_\lambda$ where
$R_\lambda$ is the onedimensional $R$-module with trivial action from
$U_R^{+}$ and $U_R^0$ acting as $\lambda$. There is a unique simple
quotient $L_R(\lambda)$ of $M_R(\lambda)$.

Let $\alpha=\alpha_i\in \Pi$. Consider the parabolic Verma module
$M_{R,i}(\lambda):=U_R(i)\tensor_{U_R^{\geq 0}}R_\lambda$, where
$U_R(i)$ is the submodule generated by $U_R^{\geq 0}$ and
$F_{\alpha_i}$. We get a map $M_{R,i}(\lambda)\to M_{R,i}^s(\lambda):=
{^{s}}((U_R(i)\tensor_{U_R(s_i)}U_R(s_i)^*)\tensor_{U_R(i)}M_{R,i}(s.\lambda))$
where the module $(U_R(i)\tensor_{U_R(s_i)}U_R(s_i)^*)$ is a
$U_R(i)$-bimodule by the similar arguments as earlier. Inducing to the
whole quantum group and using $T_w$ we get a homomorphism
\begin{equation*}
  M_R^w(\lambda)\to M_R^{ws_\alpha}(\lambda)
\end{equation*}

So we can construct a sequence of homomorphisms $\phi_1,\dots,\phi_N$
\begin{equation*}
  M_R(\lambda)\stackrel{\phi_1}{\to} M_R^{s_{i_1}}(\lambda)\stackrel{\phi_2}{\to}\cdots \cdots \stackrel{\phi_N}{\to} M_R^{w_0}(\lambda)=DM_R(\lambda).
\end{equation*}

We denote the composition by $\Psi$. Note that the image of $\Psi$
must be the unique simple quotient $L_R(\lambda)$ of $M_R(\lambda)$
since every map $M(\lambda)\to DM(\lambda)$ maps to the unique simple
quotient of $M(\lambda)$ (by the usual arguments e.g. like
in~\cite[Theorem~3.3]{Humphreys}).

First we want to consider some facts about the map
$\phi:M_R^w(\lambda)\to M_R^{ws_\alpha}(\lambda)$.  Let
$M_\alpha(\lambda)$ denote the $U_R(\mathfrak{sl}(2))$ Verma module
with highest weight $\lambda(K_\alpha)$. We will use the notation
$M_{p_\alpha}(\lambda)$ for the parabolic $U_R(i)$ Verma module
$U_R(i)\tensor_{U_R^{\geq 0}} R_\lambda$. The map $\phi$ was
constructed by first inducing the map of parabolic modules and then
using the twisting functor $T_w$.

Assume the sequence of $U_R(\mathfrak{sl}_2)$ modules
$M_\alpha(\lambda)\to M_\alpha^{s}(\lambda)\to Q_\alpha(\lambda) \to
0$ is exact (i.e. $Q_\alpha(\lambda)$ is the cokernel of the map
$M_\alpha(\lambda)\to M_\alpha^s(\lambda)$).  Inflating to the
parabolic situation we get an exact sequence $M_{p_\alpha}(\lambda)\to
M_{p_\alpha}^s(\lambda) \to Q_{p_\alpha}(\lambda) \to 0$ where
$Q_{p_\alpha}(\lambda)$ is just the inflation of $Q_\alpha(\lambda)$
to the corresponding parabolic module.

Inducing from a parabolic module to the whole module is done by
applying the functor $M\mapsto U_R\tensor_{U(i)}\nobreak M$. This is
right exact so we get the exact sequence $M_R(\lambda)\to
M_R^s(\lambda)\to Q_R(\lambda) \to 0$ where $Q_R(\lambda) =
U_R\tensor_{U_R(i)} Q_{p_\alpha}(\lambda)$.

Assume we have a finite filtration of $Q_\alpha(\lambda)$:
\begin{equation*}
  0=Q_0 \subset Q_1 \subset \cdots \subset Q_r = Q_\alpha(\lambda)
\end{equation*}
such that $Q_{i+1}/Q_i \iso L_\alpha(\mu_i)$. So we have after
inflating:
\begin{equation*}
  0 = Q_{p_\alpha,0} \subset Q_{p_\alpha,1} \subset \cdots \subset Q_{p_\alpha,r} = Q_{p_\alpha}(\lambda)
\end{equation*}
such that $Q_{p_\alpha,i+1}/Q_{p_\alpha,i} \iso L_{p_\alpha}(\mu_i)$.

That is we have short exact sequences of the form
\begin{equation*}
  0\to Q_{p_\alpha,i} \to Q_{p_\alpha,i+1} \to L_{p_\alpha}(\mu_i)\to 0.
\end{equation*}

Since induction is right exact we get the exact sequence
\begin{equation*}
  Q_{R,i} \to Q_{R,i+1} \to \bar{L_{p_\alpha}(\mu_i)}\to 0
\end{equation*}
where $Q_{R,i}$ is the induced module of $Q_{p_\alpha,i}$ and
$\bar{L_{p_\alpha}(\mu_i)}$ is the induced module of
$L_{p_\alpha}(\mu_i)$.

Starting from the top we have
\begin{equation*}
  Q_{R,r-1}\to Q_{R}(\lambda) \to \bar{L_{p_\alpha}(\mu_{r-1})} \to 0
\end{equation*}
so we see that the composition factors of $Q_R^{s_\alpha}(\lambda)$
are contained in the set of composition factors of
$\bar{L_{p_\alpha(\mu_{r-1})}}$ and the composition factors of
$Q_{R_,r-1}$. By induction we get then that the composition factors of
$Q_{R,r-1}$ are composition factors of $\bar{L_{p_\alpha}(\mu_i)}$,
$i=0,\dots,r-2$. The conclusion is that we can get a restriction on
the composition factors of $Q_R(\lambda)$ by examining the composition
factors of induced simple modules.

Let $L=L_{p_\alpha}(\mu)$ be a simple parabolic module and let
$\bar{L}$ be the induction of $L$. Then because induction is right
excact we have
\begin{equation*}
  M_R(\mu) \to \bar{L} \to 0.
\end{equation*}
So the composition factors of $\bar{L}$ are composition factors of
$M_R(\mu)$. This gives us a restriction on the composition factors of
$M_R(\lambda)$:

Use the above with $w\inv.\lambda$ in place of $\lambda$ and use the
twisting functor $T_w^R$ on the exact sequence $M_R(w\inv.\lambda)\to
M_R^s(w\inv.\lambda)\to Q_R(w\inv.\lambda) \to 0$ to get
\begin{equation*}
  M_R^w(\lambda) \to M_R^{ws}(\lambda) \to Q_R^w(\lambda) \to 0
\end{equation*}
where $Q_R^{ws}(\lambda)=T_w^R (Q_R(w\inv.\lambda))$. Add the kernel
to get the 4-term exact sequence
\begin{equation*}
  0\to K_R^{ws}(\lambda) \to M_R^w(\lambda) \to M_R^{ws}(\lambda) \to Q_R^{ws}(\lambda) \to 0
\end{equation*}
Since $\ch M_R^w(\lambda)= \ch M_R^{ws}(\lambda)$ we must have $\ch
K_R^{ws}(\lambda) = \ch Q_R^{ws}(\lambda)$.

So we have a sequence of homomorphisms $\phi_i$
\begin{equation*}
  M_R(\lambda) \stackrel{\phi_1}{\to} M_R^s(\lambda) \stackrel{\phi_2}{\to}\cdots \stackrel{\phi_N}{\to} M_R^{w_0}(\lambda)=DM_R(\lambda)
\end{equation*}
and these maps each fit into a 4-term exact sequence
\begin{equation*}
  0\to K_R^{ws}(\lambda) \to M_R^w(\lambda) \to M_R^{ws}(\lambda) \to Q_R^{ws}(\lambda) \to 0
\end{equation*}
where $\ch K_R^{ws}(\lambda) = \ch Q_R^{ws}(\lambda)$. In particular
$M_R^w(\lambda)\to M_R^{ws}(\lambda)$ is an isomorphism if the
corresponding $\mathfrak{sl}_2$ map $M_\alpha(w\inv.\lambda)\to
DM_\alpha(w\inv.\lambda)(=M_\alpha^s(w\inv.\lambda))$ is an
isomorphism.  If the $\mathfrak{sl}_2$ map is not an isomorphism then
we have a restriction on the composition factors that can get killed
by the map $M_R(w\inv.\lambda)\to M_R^{s}(w\inv.\lambda)$ by the
above. To get to the map $M_R^{w}(\lambda) \to M_R^{ws}(\lambda)$ we
use $T_w$ which is right exact so we get a restriction on the
composition factors killed by $M_R^w(\lambda)\to M_R^{ws}(\lambda)$
too:

Fix $\alpha$. From the above we know that a composition factor of
$Q_R(\lambda)$ is a composition factor of $\bar{L_{p_\alpha}(\mu)}$
for some $\mu$ where $L_\alpha(\mu)$ is a composition factor of
$M_\alpha(\lambda)$. Use this for $w\inv.\lambda$ and use $T_w$. So we
get that a composition factor of $Q_R^{ws}(\lambda)$ is a composition
factor of $T_w\bar{L_{p_\alpha}(\mu)}$ with $\mu$ as before. Since
$T_w$ is right exact we have that
\begin{equation*}
  T_w M_R(\mu) \to T_w\bar{L_{p_\alpha}(\mu)} \to 0
\end{equation*}
is exact.  Since $\ch T_w M_R(\mu) = \ch M_R(w.\mu)$ we see that a
composition factor of $Q_R^{ws}(\lambda)$ must be a composition factor
of a Verma module $M_R(w.\mu)$ where $\mu$ is such that
$L_\alpha(\mu)$ is a composition factor of $M_\alpha(w\inv.\lambda)$.

\begin{defn}
  We define a partial order on weights. We say $\mu\leq \lambda$ if
  $\mu\inv \lambda = q^{\sum_{i=1}^n a_i \alpha_i}$ for some $a_i\in
  \setN$ where $\mu\inv:U_R^0\to \setC$ is the weight with $\mu\inv (K_\alpha)=
  \mu(K_\alpha\inv)$ for all $\alpha\in \Pi$.

  For a weight $\nu$ of the form $\nu=q^{\sum_{i=1}^n a_i \alpha_i}$
  with $a_i\in \setN$ we call $\sum_{i=1}^n a_i$ the height of $\nu$.
\end{defn}

Note that for a Verma module $M(\lambda)$ we have $\mu \leq \lambda$
for all $\mu \in \wt M(\lambda)$ where $\wt M(\lambda)$ denotes the
weights of $M(\lambda)$.

\begin{defn}
  Let $\mu,\lambda\in \Lambda$. Define $\mu\uparrow_R \lambda$ to be
  the partial order induced by the following: $\mu$ is less than
  $\lambda$ if there exists a $w\in W$, $\alpha\in \Pi$ and $\nu\in
  \Lambda$ such that $\mu=w.\nu<\lambda$ and $L_\alpha(\nu)$ is a
  composition factor of $M_\alpha(w\inv.\lambda)$.

  i.e. $\mu\uparrow_R \lambda$ if there exists a sequence of weights
  $\mu=\mu_1,\dots,\mu_r=\lambda$ such that $\mu_i$ is related to
  $\mu_{i+1}$ as above.
\end{defn}

We have established the following:
\begin{prop}
  \label{prop:1}
  If $L_R(\mu)$ is a composition factor of $M_R(\lambda)$ then
  $\mu\uparrow_R \lambda$.
\end{prop}
\begin{proof}
  Choose a reduced expression of $w_0$ and construct the maps $\phi_i$
  as above.  If $L_R(\mu)$ is a composition factor of $M_R(\lambda)$
  it must be killed by one of the maps $\phi_i$ since the image of
  $\Psi$ is $L_R(\lambda)$. So $L_R(\mu)$ must be a composition factor
  of one of the modules $Q_R^w(\lambda)$. We make an induction on the
  height of $\mu\inv \lambda$. If $\mu\inv \lambda=1$ then
  $\lambda=\mu$ and we are done. Otherwise we see that $L_R(\mu)$ is a
  composition factor of one of the $Q_R^w(\lambda)$'s. But every
  composition factor of $Q_R^w(\lambda)$ is a composition factor of
  $M(\nu)$ where $\nu \uparrow_R \lambda$ and $\nu<\lambda$. Since
  $\nu<\lambda$ the height of $\mu\inv \nu$ is less then the height of
  $\mu\inv \lambda$ so we are done by induction.
\end{proof}

In the non-root of unity case $\uparrow_R$ is equivalent to the usual
strong linkage: $\mu$ is strongly linked to $\lambda$ if there exists
a sequence $\mu_i$ with $\mu=\mu_1<\mu_{2}<\cdots <\mu_r =\lambda$ and
$\mu_{i}=s_{\beta_{i}}.\mu_{i+1}$ for some positive roots $\beta_{i}$
(remember that if $\beta=w(\alpha)$ then $s_\beta=ws_\alpha w\inv$).

In the nonroot of unity case we see that $M_\alpha(w\inv.\lambda)$ is
simple if
\begin{equation*}
  q^\rho w\inv. \lambda(K_\alpha) \not \in \pm q_\alpha^{\setZ_{>0}}.
\end{equation*}
Otherwise there is one composition factor in $M_\alpha(w\inv.\lambda)$
apart from $L_\alpha(w\inv.\lambda)$, namely $L_\alpha(s_\alpha
w\inv.\lambda)$. So the composition factors of $Q_R^w$ are composition
factors of $M_R(ws_\alpha
w\inv.\lambda)=M_R(s_{w(\alpha)}.\lambda)$. Actually
$Q_R^w=M_R^{ws}(s_{w(\alpha)}.\lambda)$ in this case:

Lets consider the construction of the maps $\phi_i$ in the above. We
start with the map $M_\alpha(\lambda)\to M_\alpha^s(\lambda)$ and then
inflate to $M_{p_\alpha}(\lambda)\to M_{p_\alpha}^s(\lambda)$. In the
case where $q$ is not a root of unity it is easy to see that if
$q^\rho \lambda(K_\alpha)\not \in \pm q_\alpha^{\setZ_{>0}}$ then this
is an isomorphism and otherwise the kernel (and the cokernel) is
isomorphic to $M_{p_\alpha}(s.\lambda)$ which is a simple module. So
after inducing we get the 4 term exact sequence
\begin{equation*}
  0\to M_R(s.\lambda) \to M_R(\lambda)\to M_R^s(\lambda)\to M_R^s(s.\lambda)\to 0
\end{equation*}
since induction is exact on Verma modules. Use these observations on
$w\inv.\lambda$ and the fact that $T_w$ is exact on Verma modules and
we get a map $M_R^w(\lambda)\to M_R^{ws}(\lambda)$ which is an
isomorphism if $q^\rho \lambda(K_\alpha)\not \in \pm
q_\alpha^{\setZ_{>0}}$ and otherwise we have the 4-term exact sequence
\begin{equation*}
  0\to M_R^w(s.\lambda)\to M_R^w(\lambda)\to M_R^{ws}(\lambda)\to M_R^{ws}(s.\lambda)\to 0
\end{equation*}

\begin{thm}
  \label{thm:comp_factors}
  Let $R$ be a field (any characteristic) and let $q\in R$ be a
  non-root of unity. $R$ is an $A$-algebra by sending $v$ to $q$. Let
  $\lambda:U_q^0\to R$ be an algebra homomorphism.

  $M_R(\lambda)$ has finite Jordan-Holder length and if $L_R(\mu)$ is
  a composition factor of $M_R(\lambda)$ then $\mu \uparrow \lambda$
  where $\uparrow$ is the usual strong linkage.
\end{thm}
\begin{proof}
  This will be proved by induction over $\uparrow$. If $\lambda$ is
  anti-dominant (i.e. $q^\rho \lambda(K_\alpha)\not \in \pm
  q_\alpha^{\setZ_{>0}}$ for all $\alpha \in \Pi$) then we get that
  all the maps $\phi_i$ are isomorphisms and so $M_R(\lambda)$ is
  simple. Now assume $\lambda$ is not anti-dominant.  A composition
  factor $L_R(\mu)$ must be killed by one of the $\phi_i$'s so must be
  a compostion factor of $Q_R^w$ for some $w$. By the above
  calculations we see that if $q^\rho \lambda(K_\alpha)\not \in \pm
  q_\alpha^{\setZ_{>0}}$ then $M_R^w(\lambda)\to
  M_R^{ws_\alpha}(\lambda)$ is an isomorphism and otherwise $Q_R^w=
  M_R^{ws_\alpha}(s_\alpha.\lambda)$. By induction all the Verma
  modules with highest weight $\mu$ strongly linked to $\lambda$ has
  finite length and the composition factors are strongly linked to
  $\mu$. This finishes the induction.
\end{proof}

\bibliography{lit} \bibliographystyle{amsalpha}


\end{document}